\documentclass[11pt]{amsart}

\usepackage[margin=1.25in]{geometry}
\usepackage{amsmath, amssymb, amsthm, graphicx, enumerate, tikz, float, color}
\usepackage{mathtools, comment}
\usepackage{wasysym}
\usepackage[misc]{ifsym}
\usepackage[colorlinks]{hyperref}
\usepackage{orcidlink}

\hypersetup{citecolor=blue}
\usetikzlibrary{matrix,arrows,decorations.pathmorphing}

\newtheorem{theorem}{Theorem}[section]
\newtheorem{lemma}[theorem]{Lemma}
\newtheorem{corollary}[theorem]{Corollary}
\newtheorem{proposition}[theorem]{Proposition}

\newtheorem*{thmA}{Theorem A}
\newtheorem*{thmB}{Theorem B}
\newtheorem*{thmC}{Theorem C}
\newtheorem{hyp}[theorem]{Hypothesis}

\theoremstyle{definition}
\newtheorem{definition}[theorem]{Definition}
\newtheorem{example}[theorem]{Example}

\theoremstyle{remark}

\DeclareMathOperator{\cd}{cd}
\DeclareMathOperator{\Irr}{Irr}

\numberwithin{equation}{section}

\newcommand{\SL}[2]{\Sigma_{#1,#2}^L} 
\newcommand{\SR}[2]{\Sigma_{#1,#2}^R} 
\newcommand{\SLR}[2]{\Sigma_{#1,#2}^*} 
\newcommand{\SLi}[3]{\Sigma_{#1,#2}^{#3L}} 
\newcommand{\SRi}[3]{\Sigma_{#1,#2}^{#3R}} 
\newcommand{\SZ}[2]{\Sigma_{#1,#2}^0} 

\makeatletter
\@namedef{subjclassname@2020}{%
  \textup{2020} Mathematics Subject Classification}
\makeatother

\begin{document}

\allowdisplaybreaks

\title[Prime character degree graphs within a family (iii)]{On prime character degree graphs occurring within a family of graphs (iii)}

\author[Bissler et al.]{Mark W. Bissler \orcidlink{0000-0002-0472-0334} | Thatcher Debowski \orcidlink{0009-0001-5430-6430} | Theodore F. Hoelker \orcidlink{0009-0009-7265-2079}\\ Jacob Laubacher \orcidlink{0000-0003-0045-7951} | Lorenzo Ravaglia \orcidlink{0009-0000-4891-4854} | G. Sivanesan \orcidlink{0000-0001-7153-960X}}
\address{Mark W. Bissler \orcidlink{0000-0002-0472-0334} | Department of General Education, Western Governors University | Salt Lake City, Utah 84107, USA}
\email{\textcolor{magenta}{mark.bissler@wgu.edu}}

\address{Thatcher Debowski \orcidlink{0009-0001-5430-6430}| Department of Mathematics, Hillsdale College | Hillsdale, Michigan 49242, USA}
\email{\textcolor{magenta}{tdebowski@hillsdale.edu}}

\address{Theodore F. Hoelker \orcidlink{0009-0009-7265-2079} | Department of Mathematics, Hillsdale College | Hillsdale, Michigan 49242, USA}
\email{\textcolor{magenta}{thoelker@hillsdale.edu}}

\address{Jacob Laubacher \orcidlink{0000-0003-0045-7951} | Department of Mathematics, Hillsdale College | Hillsdale, Michigan 49242, USA}
\email{\textcolor{magenta}{jlaubacher@hillsdale.edu}}

\address{Lorenzo Ravaglia \orcidlink{0009-0000-4891-4854} | Department of Mathematics, Hillsdale College | Hillsdale, Michigan 49242, USA}
\email{\textcolor{magenta}{lravaglia@hillsdale.edu}}

\address{G. Sivanesan \orcidlink{0000-0001-7153-960X} | Department of Mathematics, Government College of Engineering | Salem 636001, Tamil Nadu, India}
\email{\textcolor{magenta}{sivanesan@gceslem.edu.in}}

\date{\today}

\subjclass[2020]{Primary 20D10; Secondary 20C15, 05C75}

\keywords{Character degree graphs, solvable groups, families of graphs\\\indent\emph{Corresponding author.} Jacob Laubacher \Letter~\href{mailto:jlaubacher@hillsdale.edu}{jlaubacher@hillsdale.edu} \phone~(517) 607-3285.}

\begin{abstract}
We conclude the classification work done in the two previous papers of the same name. Here we add flexibility to the construction, thereby viewing the graphs in full generality. Our goal, as ever, is to determine which graphs do or do not occur as the prime character degree graph of a solvable group.
\end{abstract}

\maketitle

\section{Introduction}

Let $G$ be a finite solvable group, and denote $\Irr(G)$ as the set of irreducible characters of $G$. The set of character degrees of $G$ is then $\cd(G)=\{\chi(1)~|~\chi\in\Irr(G)\}$. Our main focus of study is the prime character degree graph of $G$, which is denoted as $\Delta(G)$. The vertex set of this simple graph is written as $\rho(\Delta(G))$ (or even $\rho(G)$ for short), and it consists of all the prime numbers that divide some character degree in $\cd(G)$. There is an edge between two vertices $p$ and $q$ in the graph $\Delta(G)$ if there exists some character degree $a\in\cd(G)$ such that $pq\mid a$. Notice that in this context, the words vertex and prime are synonymous. 

Prime character degree graphs have been a constant object of study, ranging through a wealth of topics like its metric dimension (see \cite{CSSTL}, for instance), disconnected components (like \cite{L} or \cite{LS}), regularity (see \cite{M2}), Eulerian paths (\cite{SST}), its diameter (see such papers like \cite{CDPS}, \cite{DHS}, \cite{L2}, \cite{L2.5}, and \cite{S}) associated eigenvalues (\cite{SS}), complete vertices (see \cite{M}), or even a broad overview (like \cite{L4}) or classic background (\cite{H2} and \cite{I}). Classifying whether a graph can or cannot occur as the prime character degree graph of a solvable group has been a point of interest, where we note that Zhang has cataloged full results for graphs with one, two, three, and four vertices in \cite{Z}. The graphs with five, six, seven, and eight vertices have all been studied in \cite{L3}, \cite{BLL}, \cite{LMS}, and \cite{LS2}, respectively.

More recently, work on families of graphs have been done, starting with the collections introduced in \cite{BL} and \cite{BLL2}. In fact, the object of investigation for the paper at hand is to finish the classification on the family of graphs first studied in \cite{LM} and then continued in \cite{DLM}. We recall the construction of the graph $\SLR{k}{n}$ in its full generality, where there are two main variants that we make note of: a ``left" version of the graph which we denote by $\SLi{k}{n}{m}$, and a ``right" version denoted by $\SRi{k}{n}{m}$. We then pick arbitrary integers $k$ and $m$ such that $k\geq1$ and $m\geq0$, and we enforce the restriction that $n$ is an integer such that $1\leq n\leq k$. Then, the graph $\SLR{k}{n}$ consists of three distinct subgraphs $A$, $B$, and $C$ which satisfy the following:
\begin{enumerate}[(i)]
    \item $A$ is a complete graph on the $k$ vertices $a_1, a_2, \ldots, a_k$,
    \item $B$ is a complete graph on the $k+n$ vertices $b_1, b_2, \ldots, b_{k+n}$,
    \item $C$ is a complete graph on the $m$ vertices $c_1, c_2, \ldots, c_m$,
    \item $\rho(A)$, $\rho(B)$, and $\rho(C)$ are all pairwise disjoint,
    \item there is an edge between $a_i$ and $b_i$ for all $1\leq i\leq k$,
    \item there is an edge between $a_i$ and $b_{k+i}$ for all $1\leq i\leq n$,
    \item there is an edge between $a_i$ and $c_j$ for all $1\leq i\leq k$ and $1\leq j\leq m$ in $\SLi{k}{n}{m}$,
    \item there is an edge between $b_i$ and $c_j$ for all $1\leq i\leq k+n$ and $1\leq j\leq m$ in $\SRi{k}{n}{m}$, and
    \item there are no other edges in $\SLR{k}{n}$ other than what is described above.
\end{enumerate}

The notation of $\SLR{k}{n}$, which is clunky and unwieldy at first glance, becomes natural and convenient in practice. Reading the graph from left to right, the subscript of $k$ always counts the number of ``$a$" vertices, and the subscript $n$ counts the number of ``one-to-two" edge-mappings spanning $A$ to $B$. The superscript of $L$ or $R$ dictate whether the ``$c$" vertices (of which there are $m$ of them) are attached to the left (read, $A$) or right ($B$) of the graph. Work on this family has been done before. In the first part of this paper (\cite{LM}), results were garnered specifically for the case of $m=1$ and associated solely for the left graphs. In other words, \cite{LM} focused on the family of graphs $\{\SLi{k}{n}{1}\}$. On the other hand, the second part (\cite{DLM}) yielded results for $\{\SRi{k}{n}{1}$\}.

The goal of this paper is to finish the classification for the family $\{\SLR{k}{n}\}$ in full generality by investigating the cases of $m=0$ and $m\geq2$. After providing all the necessary preliminaries in Section \ref{secpre}, we study the graph $\SLi{k}{n}{m}$ for $m\geq2$ in Section \ref{secleft} and the graph $\SRi{k}{n}{m}$ for $m\geq2$ in Section \ref{secright}. Notice that by taking $m=0$, we force $\rho(C)=\varnothing$, thereby having no ``$c$" vertices on the left nor right. We investigate this phenomenon in Section \ref{seczero}. We finish with three main results, which we label Theorem A, Theorem B, and Theorem C, that appear at the end of Section \ref{secleft}, Section \ref{secright}, and Section \ref{seczero}, respectively.

\section{Preliminaries}\label{secpre}

We present here the important relevant background information that is used most frequently in this paper. Our goal is to stay as self-contained as possible. We start, therefore, with a quick note about notation and convention. We will often be referring to a graph that that either can or cannot occur as the prime character degree graph of a solvable group. We occasionally abbreviate the above to instead say that a graph \emph{occurs} or \emph{does not occur}, respectively. Following this vein, to say that a graph is \emph{non-occurring} clearly implies it \emph{does not occur}. Further, we frequently mention removing a vertex $p$ and all incident edges from a graph $\Gamma$, which we notate as $\Gamma[p]$. Similarly, having an edge between $p$ and $q$ (denoted $\epsilon(p,q)$) removed from a graph $\Gamma$ can be written as $\Gamma[\epsilon(p,q)]$. Finally, we follow convention and use the notation of $K_n$ to refer to a complete graph on $n$ vertices.

We next address the landmark result of P\'alfy's, which is the most basic tool in the study of solvable groups. It is helpful in determining if a graph \emph{cannot} occur as the prime character degree graph of a solvable group.

\begin{lemma}[P\'alfy's condition from \cite{P}]\label{PC}
Let $G$ be a solvable group and let $\pi$ be a set of primes contained in $\Delta(G)$. If $|\pi|=3$, then there exists an irreducible character of $G$ with degree divisible by at least two primes from $\pi$. (In other words, any three vertices of the prime character degree graph of a solvable group span at least one edge.)
\end{lemma}

Another way to process P\'alfy's condition is to recognize that a prime character degree graph corresponding to a finite solvable group has no triangle ($3$-cycle) in its complement graph. We note that in 2018 in \cite{A}, P\'alfy's condition was generalized to accommodate any odd cycle in the complement graph.

Our next tool, made formal in \cite{BLL2}, helps classify individual vertices of the graphs that we investigate.

\begin{definition}\label{adstrong}(\cite{BLL2})
Let $\Gamma$ be a graph and $p$ a vertex of $\Gamma$. Consider the following three conditions:
\begin{enumerate}[(i)]
    \item the subgraph of $\Gamma$ obtained by removing $p$ and all edges incident to $p$ is non-occurring,
    \item all of the subgraphs of $\Gamma$ obtained by removing one or more of the edges incident to $p$ are non-occurring,
    \item all of the subgraphs of $\Gamma$ obtained by removing $p$, the edges incident to $p$, and one or more of the edges between two adjacent vertices of $p$ are non-occurring.
\end{enumerate}
If $p$ satisfies conditions (i) and (ii), then $p$ is said to be an \textbf{admissible} vertex. If $p$ satisfies all three conditions, then $p$ is said to be a \textbf{strongly admissible} vertex of $\Gamma$.
\end{definition}

There are heavy consequences if every vertex is admissible (Theorem \ref{admissible} below), and albeit less substantial, there are still crucial results for our arguments if even a single vertex is admissible (see Lemma \ref{bigoh}) or strongly admissible (see Lemma \ref{strong}).

\begin{theorem}\label{admissible}\emph{(\cite{BLL2})}
Assume that $\Gamma$ is a non-empty graph. If $\Gamma$ is a graph in which every vertex is admissible, then $\Gamma$ is not $\Delta(G)$ for any solvable group $G$.
\end{theorem}

\begin{lemma}\label{bigoh}\emph{(\cite{BLL2})}
Let $G$ be a solvable group, and suppose $p$ is an admissible vertex of $\Delta(G)$. For every proper normal subgroup $N$ of $G$, suppose that $\Delta(N)$ is a proper subgraph of $\Delta(G)$. Then $O^p(G)=G$.
\end{lemma}

\begin{lemma}\label{strong}\emph{(\cite{BLL2})}
Let $G$ be a solvable group and assume that $p$ is a prime whose vertex is a strongly admissible vertex of $\Delta(G)$. For every proper normal subgroup $N$ of $G$, suppose that $\Delta(G/N)$ is a proper subgraph of $\Delta(G)$. Then a Sylow $p$-subgroup of $G$ is not normal.
\end{lemma}

One of our strategies for eliminating graphs will be to first conclude that there are no normal nonabelian Sylow subgroups associated with some prime (vertex). In fact, we see that conclusion above for a strongly admissible vertex. However, if a vertex is not strongly admissible, we need a way to still garner that conclusion. One method to do this is to satisfy a five-part technical hypothesis in \cite{BLL3}, and another is to satisfy the lemma below.

\begin{lemma}\label{pi}\emph{(\cite{BLL2})}
Let $\Gamma$ be a graph satisfying P\'alfy's condition. Let $q$ be a vertex of $\Gamma$, and denote $\pi$ to be the set of vertices of $\Gamma$ adjacent to $q$, and $\rho$ to be the set of vertices of $\Gamma$ not adjacent to $q$. Assume that $\pi$ is the disjoint union of nonempty sets $\pi_1$ and $\pi_2$, and assume that no vertex in $\pi_1$ is adjacent in $\Gamma$ to any vertex in $\pi_2$. Let $v$ be a vertex in $\pi_2$ adjacent to an admissible vertex $s$ in $\rho$. Furthermore, assume there exists another vertex $w$ in $\rho$ that is not adjacent to $v$.

Let $G$ be a solvable group such that $\Delta(G)=\Gamma$, and assume that for every proper normal subgroup $N$ of $G$, $\Delta(N)$ is a proper subgraph of $\Delta(G)$. Then a Sylow $q$-subgroup of $G$ for the prime associated to $q$ is not normal.
\end{lemma}

We next recall a result from \cite{BLL2} which allows us a final contradiction to help eliminate a graph. Notice the use of admissible vertices.

\begin{proposition}\label{final}\emph{(\cite{BLL2})}
Let $\Gamma$ be a graph satisfying P\'alfy's condition with $n\geq5$ vertices. Also, assume there exist distinct vertices $a$ and $b$ of $\Gamma$ such that $a$ is adjacent to an admissible vertex $c$, $b$ is not adjacent to $c$, and $a$ is not adjacent to an admissible vertex $d$.

Let $G$ be a solvable group and suppose for all proper normal subgroups $N$ of $G$ we have that $\Delta(N)$ and $\Delta(G/N)$ are proper subgraphs of $\Gamma$. Let $F$ be the Fitting subgroup of $G$ and suppose that $F$ is minimal normal in $G$. Then $\Gamma$ is not the prime character degree graph of any solvable group.
\end{proposition}

In our study of the family of graphs $\{\SLR{k}{n}\}$, we frequently rely on results from other previously investigated collections of graphs. Chief among them is the family of graphs introduced in \cite{BL} denoted by $\{\Gamma_{k,t}\}$, whose construction is as follows:

Take complete graphs $A$ and $B$ such that $A$ has $k$ vertices $a_1, a_2, \ldots, a_k$ and $B$ has $t$ vertices $b_1, b_2, \ldots, b_t$. Without loss of generality, we take $k\geq t$. Furthermore, set $\rho(\Gamma_{k,t})=\rho(A)\cup\rho(B)$ where we note that we can always enforce $\rho(A)\cap\rho(B)=\varnothing$. There is an edge in $\Gamma_{k,t}$ between vertices $p$ and $q$ if
\begin{enumerate}[(i)]
    \item $p,q\in\rho(A)$,
    \item $p,q\in\rho(B)$, or
    \item $p=a_i\in\rho(A)$ and $q=b_i\in\rho(B)$ for some $1\leq i\leq t$.
\end{enumerate}

\begin{theorem}\label{KT}\emph{(\cite{BL})}
The graph $\Gamma_{k,t}$ occurs as the prime character degree graph of a solvable group precisely when $t=1$ or $k=t=2$. Otherwise $\Gamma_{k,t}$ does not occur as the prime character degree graph of any solvable group, nor does any proper connected subgraph of $\Gamma_{k,t}$ with the same vertex set whenever $k\geq t\geq2$.
\end{theorem}

Our final frequently-used tool helps tame graphs with a diameter of three. The following, presented in \cite{S}, allows us to partition the vertex set $\rho(\Gamma)$ of a diameter three graph $\Gamma$ into four nonempty pairwise disjoint sets $\rho_1$, $\rho_2$, $\rho_3$, and $\rho_4$. We refer the reader to \cite{S} for more subtle detail, but one can always find a labeling that guarantees the following:

\begin{theorem}\label{diameter}\emph{(\cite{S})}
Let $G$ be a solvable group where $\Delta(G)$ has diameter three. One then has the following:
\begin{enumerate}[(i)]
    \item\label{diam1} $|\rho_3|\geq3$,
    \item\label{diam2} $|\rho_1\cup\rho_2|\leq|\rho_3\cup\rho_4|$,
    \item\label{diam3} if $|\rho_1\cup\rho_2|=n$, then $|\rho_3\cup\rho_4|\geq2^n$, and
    \item\label{diam4} $G$ has a normal Sylow $p$-subgroup for exactly one prime $p\in\rho_3$.
\end{enumerate}
\end{theorem}

\section{The Left Family}\label{secleft}

Our aim in this section is to classify the family $\{\SLi{k}{n}{m}\}$. Work on this family has been done in \cite{LM}, namely on the graphs $\SL{k}{n}$. One can see the result below:

\begin{theorem}\label{Lfam}\emph{(\cite{LM})}
The graph $\SL{k}{n}$ occurs as the prime character degree graph of a solvable group when $(k,n)=(1,1)$. Otherwise, $\SL{k}{n}$ does not occur as the prime character degree graph of any solvable group.
\end{theorem}

In particular, we see that Theorem \ref{Lfam} is nothing more than the case of $m=1$. That is, $\SLi{k}{n}{1}=\SL{k}{n}$. Our goal here is to classify the graphs for $m\geq2$.

\subsection{The case of $k=1$}

The scenario of $k=1$ brings to light some unsolved problems in number theory. One can see Figure \ref{fig11mL} for examples of such graphs, which are addressed in Proposition \ref{prop11mL}.

\begin{figure}[htb]
    \centering
$
\begin{tikzpicture}[scale=2]
\node (uc1) at (0,1) {$c_1$};
\node (uc2) at (0,0) {$c_2$};
\node (ua1) at (0.5,0.5) {$a_1$};
\node (ub1) at (1,1) {$b_1$};
\node (ub2) at (1,0) {$b_2$};
\path[font=\small,>=angle 90]
(uc1) edge node [right] {$ $} (uc2)
(uc1) edge node [right] {$ $} (ua1)
(uc2) edge node [right] {$ $} (ua1)
(ua1) edge node [right] {$ $} (ub1)
(ua1) edge node [right] {$ $} (ub2)
(ub1) edge node [right] {$ $} (ub2);
\node (vc1) at (2.25,1) {$c_1$};
\node (vc2) at (2.25,0) {$c_2$};
\node (vc3) at (1.75,0.5) {$c_3$};
\node (va1) at (2.75,0.5) {$a_1$};
\node (vb1) at (3.25,1) {$b_1$};
\node (vb2) at (3.25,0) {$b_2$};
\path[font=\small,>=angle 90]
(vc1) edge node [right] {$ $} (vc2)
(vc1) edge node [right] {$ $} (va1)
(vc2) edge node [right] {$ $} (va1)
(va1) edge node [right] {$ $} (vb1)
(va1) edge node [right] {$ $} (vb2)
(vb1) edge node [right] {$ $} (vb2)
(vc1) edge node [right] {$ $} (vc3)
(vc3) edge node [right] {$ $} (vc2)
(va1) edge node [right] {$ $} (vc3);
\node (wc1) at (4.5,1) {$c_1$};
\node (wc2) at (4.5,0) {$c_2$};
\node (wc3) at (4,0.8) {$c_3$};
\node (wc4) at (4,0.2) {$c_4$};
\node (wa1) at (5,0.5) {$a_1$};
\node (wb1) at (5.5,1) {$b_1$};
\node (wb2) at (5.5,0) {$b_2$};
\path[font=\small,>=angle 90]
(wc1) edge node [right] {$ $} (wc2)
(wc1) edge node [right] {$ $} (wa1)
(wc2) edge node [right] {$ $} (wa1)
(wa1) edge node [right] {$ $} (wb1)
(wa1) edge node [right] {$ $} (wb2)
(wb1) edge node [right] {$ $} (wb2)
(wc1) edge node [right] {$ $} (wc3)
(wc3) edge node [right] {$ $} (wc2)
(wa1) edge node [right] {$ $} (wc3)
(wc3) edge node [right] {$ $} (wc4)
(wc2) edge node [right] {$ $} (wc4)
(wa1) edge node [right] {$ $} (wc4)
(wc1) edge node [right] {$ $} (wc4);
\node (xc1) at (1.375,-0.5) {$c_1$};
\node (xc2) at (1.375,-1.5) {$c_2$};
\node (xc3) at (.875,-0.5) {$c_3$};
\node (xc4) at (.875,-1.5) {$c_4$};
\node (xc5) at (0.375,-1) {$c_5$};
\node (xa1) at (1.875,-1) {$a_1$};
\node (xb1) at (2.375,-0.5) {$b_1$};
\node (xb2) at (2.375,-1.5) {$b_2$};
\path[font=\small,>=angle 90]
(xc1) edge node [right] {$ $} (xc2)
(xc1) edge node [right] {$ $} (xa1)
(xc2) edge node [right] {$ $} (xa1)
(xa1) edge node [right] {$ $} (xb1)
(xa1) edge node [right] {$ $} (xb2)
(xb1) edge node [right] {$ $} (xb2)
(xc1) edge node [right] {$ $} (xc3)
(xc3) edge node [right] {$ $} (xc2)
(xa1) edge node [right] {$ $} (xc3)
(xc3) edge node [right] {$ $} (xc4)
(xc2) edge node [right] {$ $} (xc4)
(xa1) edge node [right] {$ $} (xc4)
(xc1) edge node [right] {$ $} (xc4)
(xc3) edge node [right] {$ $} (xc4)
(xc1) edge node [right] {$ $} (xc4)
(xc3) edge node [right] {$ $} (xc5)
(xc4) edge node [right] {$ $} (xc5)
(xc5) edge node [right] {$ $} (xa1)
(xc2) edge node [right] {$ $} (xc3)
(xc2) edge node [right] {$ $} (xc5)
(xc1) edge node [right] {$ $} (xc5);
\node (yc1) at (4.125,-0.6) {$c_1$};
\node (yc2) at (4.125,-1.4) {$c_2$};
\node (yc3) at (3.625,-0.5) {$c_3$};
\node (yc4) at (3.625,-1.5) {$c_4$};
\node (yc5) at (3.125,-0.75) {$c_5$};
\node (yc6) at (3.125,-1.25) {$c_6$};
\node (ya1) at (4.625,-1) {$a_1$};
\node (yb1) at (5.125,-0.5) {$b_1$};
\node (yb2) at (5.125,-1.5) {$b_2$};
\path[font=\small,>=angle 90]
(yc1) edge node [right] {$ $} (yc2)
(yc1) edge node [right] {$ $} (ya1)
(yc2) edge node [right] {$ $} (ya1)
(ya1) edge node [right] {$ $} (yb1)
(ya1) edge node [right] {$ $} (yb2)
(yb1) edge node [right] {$ $} (yb2)
(yc1) edge node [right] {$ $} (yc3)
(yc3) edge node [right] {$ $} (yc2)
(ya1) edge node [right] {$ $} (yc3)
(yc3) edge node [right] {$ $} (yc4)
(yc2) edge node [right] {$ $} (yc4)
(ya1) edge node [right] {$ $} (yc4)
(yc1) edge node [right] {$ $} (yc4)
(yc3) edge node [right] {$ $} (yc4)
(yc1) edge node [right] {$ $} (yc4)
(yc3) edge node [right] {$ $} (yc5)
(yc4) edge node [right] {$ $} (yc5)
(yc5) edge node [right] {$ $} (ya1)
(yc2) edge node [right] {$ $} (yc3)
(yc2) edge node [right] {$ $} (yc5)
(yc1) edge node [right] {$ $} (yc5)
(yc1) edge node [right] {$ $} (yc6)
(ya1) edge node [right] {$ $} (yc6)
(yc2) edge node [right] {$ $} (yc6)
(yc4) edge node [right] {$ $} (yc6)
(yc6) edge node [right] {$ $} (yc3)
(yc5) edge node [right] {$ $} (yc6);
\end{tikzpicture}
$
    \caption{The graphs $\SLi{1}{1}{m}$ for $2\leq m\leq6$}
    \label{fig11mL}
\end{figure}
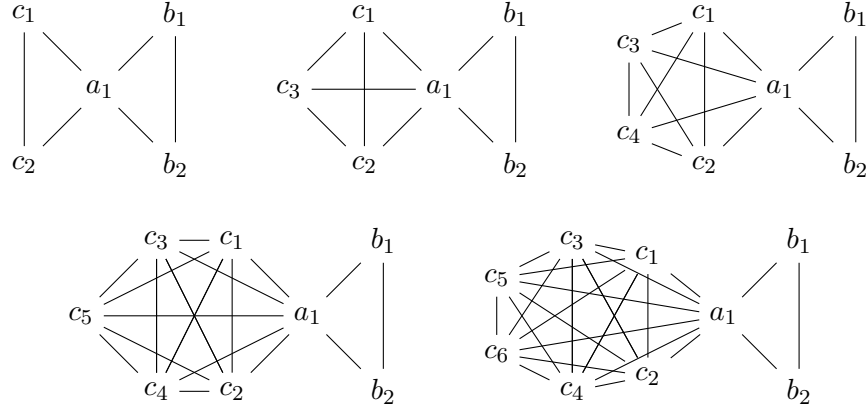

\begin{proposition}\label{prop11mL}
The graph $\SLi{1}{1}{m}$ occurs as the prime character degree graph of some solvable group for all $2\leq m\leq35$.
\end{proposition}
\begin{proof}
First note that for $m=2$, the graph $\SLi{1}{1}{2}$ is exactly the bowtie graph with five vertices that Lewis constructs in \cite{L3}, verifying it does indeed occur as the prime character degree graph of some finite solvable group.

Next, let $m\in\mathbb{N}$ such that $3\leq m\leq35$. We follow an approach similar to that done in \cite{L3} (see also \cite{BLL}, \cite{LMS}, and \cite{LS2}). For each appropriate $m$, we start by considering the finite field of order $2^{p^aq^b}$ acted on by its full multiplication group followed by its Galois group, where $p$, $a$, $q$, and $b$ are all defined based on $m$ given in Figure \ref{figpandq}. One can verify that $2^{p^aq^b}-1$ is the product of powers of $m$ distinct primes, all of which are different than $p$ and $q$. (One can do this using Mathematica, or an online database, like factordb, for example.) In notation, we now have the finite solvable group $G_m$ given as the following:
$$
G_m=\mathbb{F}_{2^{p^aq^b}}\rtimes\mathbb{F}_{2^{p^aq^b}}^\times\rtimes\text{Gal}(\mathbb{F}_{2^{p^aq^b}}/\mathbb{F}_2),
$$
and one can easily check that the set of character degrees for $G_m$ is
$$
\cd(G_m)=\{1,p^i,q^j,p^iq^j,2^{p^aq^b}-1~|~1\leq i\leq a\text{~and~}1\leq j\leq b\}.
$$
Therefore, $\Delta(G_m)$ can be seen as the disconnected graph where one component is $K_2$ with the vertex set $\{p,q\}$, and the other is $K_m$ where the $m$ vertices come from the $m$ distinct primes that are divisors of $2^{p^aq^b}-1$.

Next, we recall that the singleton $K_1$ occurs as the prime character degree graph of some solvable group $G_1$. In notation, we have that $\Delta(G_1)=K_1$.

We now consider the direct product of $G_1$ and $G_m$. Notice that $\Delta(G_1\times G_m)=\SLi{1}{1}{m}$, which is what we wanted.
\end{proof}

\begin{figure}[htb]
    \centering
    \resizebox{\textwidth}{!}{
\begin{tabular}{c|c|c|c|c|c|||c|c|c|c|c|c|||c|c|c|c|c|c}
$m$ & $p$ & $a$ & $q$ & $b$ & $2^{p^aq^b}-1$ & $m$ & $p$ & $a$ & $q$ & $b$ & $2^{p^aq^b}-1$ & $m$ & $p$ & $a$ & $q$ & $b$ & $2^{p^aq^b}-1$\\
\hline\hline\hline
3 & 2 & 1 & 5 & 1 & $2^{10}-1$ & 14 & 7 & 1 & 163 & 1 & $2^{1141}-1$ & 25 & 2 & 5 & 107 & 1 & $2^{3424}-1$\\
4 & 5 & 1 & 7 & 1 & $2^{35}-1$ & 15 & 3 & 2 & 37 & 1 & $2^{333}-1$ & 26 & 2 & 4 & 127 & 1 & $2^{2032}-1$\\
5 & 3 & 1 & 17 & 1 & $2^{51}-1$ & 16 & 2 & 2 & 97 & 1 & $2^{388}-1$ & 27 & 2 & 4 & 7 & 2 & $2^{784}-1$\\
6 & 11 & 1 & 13 & 1 & $2^{143}-1$ & 17 & 5 & 3 & 13 & 1 & $2^{1625}-1$ & 28 & 2 & 2 & 31 & 2 & $2^{3844}-1$\\
7 & 7 & 1 & 23 & 1 & $2^{161}-1$ & 18 & 3 & 3 & 53 & 1 & $2^{1431}-1$ & 29 & 2 & 6 & 19 & 1 & $2^{1216}-1$\\
8 & 3 & 1 & 53 & 1 & $2^{159}-1$ & 19 & 2 & 3 & 59 & 1 & $2^{472}-1$ & 30 & 2 & 6 & 67 & 1 & $2^{4288}-1$\\
9 & 13 & 1 & 23 & 1 & $2^{299}-1$ & 20 & 3 & 3 & 37 & 1 & $2^{999}-1$ & 31 & 2 & 6 & 29 & 1 & $2^{1856}-1$\\
10 & 2 & 1 & 113 & 1 & $2^{226}-1$ & 21 & 3 & 4 & 13 & 1 & $2^{1053}-1$ & 32 & 2 & 4 & 163 & 1 & $2^{2608}-1$\\
11 & 23 & 1 & 29 & 1 & $2^{667}-1$ & 22 & 2 & 3 & 157 & 1 & $2^{1256}-1$ & 33 & 2 & 4 & 11 & 2 & $2^{1936}-1$\\
12 & 13 & 1 & 73 & 1 & $2^{949}-1$ & 23 & 2 & 6 & 11 & 1 & $2^{704}-1$ & 34 & 2 & 4 & 13 & 2 & $2^{2704}-1$\\
13 & 31 & 1 & 47 & 1 & $2^{1457}-1$ & 24 & 2 & 2 & 23 & 2 & $2^{2116}-1$ & 35 & 2 & 5 & 7 & 2 & $2^{1568}-1$\\
\end{tabular}
    }
    \caption{The construction of the group $G_m$}
    \label{figpandq}
\end{figure}

Concerning the groups $G_m$ above, we note that some of these have indeed been constructed before, specifically $3\leq m\leq 6$. In fact, Figure \ref{figpandq} documents the exact primes used for the respective constructions that the corresponding authors have used for $m=3$ (as seen in \cite{L3}), for $m=4$ (\cite{BLL}), for $m=5$ (\cite{LMS}), and for $m=6$ (\cite{LS2}). The rest are new constructions. See Example \ref{exm7} below for an explicit description for $m=7$, if interested. Moreover, for $m>35$, we note that the graph $\SLi{1}{1}{m}$ \emph{possibly} occurs. If one has enough computing power, one can conceivably find more that do occur. Mercifully, we stopped at $35$.

\begin{example}\label{exm7}
Here we present the explicit construction of the finite solvable group from Proposition \ref{prop11mL} corresponding to $m=7$. Following the entry for $m=7$ in Figure \ref{figpandq}, we use the prime numbers $7$ and $23$, and note that $7\cdot23=161$. Therefore, we are considering the finite field of order $2^{161}$ acted on by its full multiplication group and then its Galois group, which we ultimately denote as $G_7$. We see that $2^{161}-1$ is the product of seven distinct prime numbers, all of which are different that $7$ and $23$. In fact,
$$
2^{161}-1=47\cdot127\cdot1289\cdot178481\cdot3188767\cdot45076044553\cdot14808607715315782481,
$$
where for ease of notation, we can label those prime factors as $r_i$ for $1\leq i\leq7$. We then have that $\cd(G_7)=\{1,7,23,161,2^{161}-1\}$, which yields a disconnected prime character degree graph with two components, one of which will be $K_2$ (specifically, the complete graph on the vertex set $\{7,23\}$), and the other $K_7$ (the complete graph on the vertex set $\{r_1,r_2,r_3,r_4,r_5,r_6,r_7\}$). Following the notation above where $G_1$ is the finite solvable group whose prime character degree graph is the singleton (which we can guarantee is different than all of $p$, $q$, and $r_i$ for all $1\leq i\leq7$, denoted with a bullet on the graph below), we then consider $\Delta(G_1\times G_7)$ as seen Figure \ref{fig117L}, which is exactly the graph $\SLi{1}{1}{7}$.
\begin{figure}[htb]
    \centering
$
\begin{tikzpicture}[scale=2]
\node (r1) at (0,.5) {$r_1$};
\node (r2) at (.5,.9) {$r_2$};
\node (r3) at (.5,.1) {$r_3$};
\node (r4) at (1,1) {$r_4$};
\node (r5) at (1,0) {$r_5$};
\node (r6) at (1.5,.9) {$r_6$};
\node (r7) at (1.5,.1) {$r_7$};
\node (x) at (2,.5) {$\bullet$};
\node (p) at (2.5,1) {$7$};
\node (q) at (2.5,0) {$23$};
\path[font=\small,>=angle 90]
(p) edge node [right] {$ $} (q)
(r1) edge node [right] {$ $} (r2)
(r1) edge node [right] {$ $} (r3)
(r1) edge node [right] {$ $} (r4)
(r1) edge node [right] {$ $} (r5)
(r1) edge node [right] {$ $} (r6)
(r1) edge node [right] {$ $} (r7)
(r2) edge node [right] {$ $} (r3)
(r2) edge node [right] {$ $} (r4)
(r2) edge node [right] {$ $} (r5)
(r2) edge node [right] {$ $} (r6)
(r2) edge node [right] {$ $} (r7)
(r3) edge node [right] {$ $} (r4)
(r3) edge node [right] {$ $} (r5)
(r3) edge node [right] {$ $} (r6)
(r3) edge node [right] {$ $} (r7)
(r4) edge node [right] {$ $} (r5)
(r4) edge node [right] {$ $} (r6)
(r4) edge node [right] {$ $} (r7)
(r5) edge node [right] {$ $} (r6)
(r5) edge node [right] {$ $} (r7)
(r6) edge node [right] {$ $} (r7)
(x) edge node [right] {$ $} (p)
(x) edge node [right] {$ $} (q)
(x) edge node [right] {$ $} (r1)
(x) edge node [right] {$ $} (r2)
(x) edge node [right] {$ $} (r3)
(x) edge node [right] {$ $} (r4)
(x) edge node [right] {$ $} (r5)
(x) edge node [right] {$ $} (r6)
(x) edge node [right] {$ $} (r7);
\end{tikzpicture}
$
    \caption{The graph $\Delta(G_1\times G_7)=\SLi{1}{1}{7}$}
    \label{fig117L}
\end{figure}
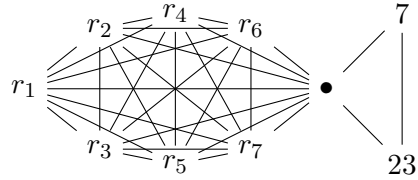
\end{example}

\subsection{The case of $k=2$}

Notice that in the case of $k=2$, there are only two options for $n$ due to the restriction that $1\leq n\leq k$. For ease of argument, we therefore treat the cases of $n=1$ (see Figure \ref{fig21mL}) and $n=2$ (see Figure \ref{fig22mL}) separately.

For the case of $n=1$, we will proceed by induction on $m$. Let $m\in\mathbb{N}$ such that $m\geq2$. Notice that for the base case of $m=2$, the graph $\SLi{2}{1}{2}$ was shown not to occur as the prime character degree graph of any solvable group in \cite{LMS}. In that paper, following the label given by the authors, the graph $\SLi{2}{1}{2}$ is exactly the graph $C_{18}$. We then state the inductive hypothesis formally as follows:

\begin{hyp}\label{hyp21mL}
Given any integer $t\geq2$, we assume that the graph $\SLi{2}{1}{t}$ does not occur as the prime character degree graph of any solvable group.
\end{hyp}

For the inductive step, we aim to show that the graph $\SLi{2}{1}{(t+1)}$ also does not occur. We will follow a similar argument to that done in \cite{LMS}, which consists of a series of lemmas, the first of which shows that all but one of the vertices are strongly admissible.

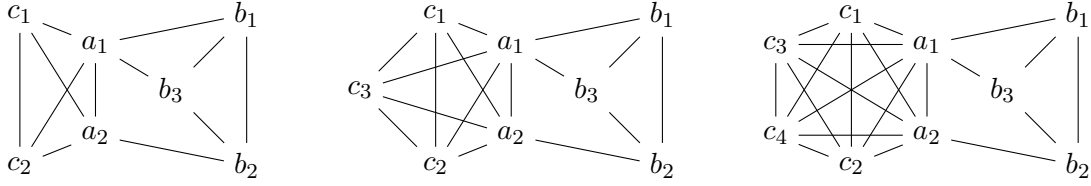
\begin{figure}[htb]
    \centering
$
\begin{tikzpicture}[scale=2]
\node (xc1) at (0,1) {$c_1$};
\node (xc2) at (0,0) {$c_2$};
\node (xa1) at (.5,.8) {$a_1$};
\node (xa2) at (.5,.2) {$a_2$};
\node (xb3) at (1,.5) {$b_3$};
\node (xb1) at (1.5,1) {$b_1$};
\node (xb2) at (1.5,0) {$b_2$};
\path[font=\small,>=angle 90]
(xc1) edge node [right] {$ $} (xc2)
(xa1) edge node [right] {$ $} (xa2)
(xb1) edge node [right] {$ $} (xb2)
(xb1) edge node [right] {$ $} (xb3)
(xb2) edge node [right] {$ $} (xb3)
(xc1) edge node [right] {$ $} (xa1)
(xc1) edge node [right] {$ $} (xa2)
(xc2) edge node [right] {$ $} (xa1)
(xc2) edge node [right] {$ $} (xa2)
(xa1) edge node [right] {$ $} (xb1)
(xa2) edge node [right] {$ $} (xb2)
(xa1) edge node [right] {$ $} (xb3);
\node (yc1) at (2.75,1) {$c_1$};
````    \node (yc2) at (2.75,0) {$c_2$};
\node (yc3) at (2.25,.5) {$c_3$};
\node (ya1) at (3.25,.8) {$a_1$};
\node (ya2) at (3.25,.2) {$a_2$};
\node (yb3) at (3.75,.5) {$b_3$};
\node (yb1) at (4.25,1) {$b_1$};
\node (yb2) at (4.25,0) {$b_2$};
\path[font=\small,>=angle 90]
(yc1) edge node [right] {$ $} (yc2)
(ya1) edge node [right] {$ $} (ya2)
(yb1) edge node [right] {$ $} (yb2)
(yb1) edge node [right] {$ $} (yb3)
(yb2) edge node [right] {$ $} (yb3)
(yc1) edge node [right] {$ $} (ya1)
(yc1) edge node [right] {$ $} (ya2)
(yc2) edge node [right] {$ $} (ya1)
(yc2) edge node [right] {$ $} (ya2)
(ya1) edge node [right] {$ $} (yb1)
(ya2) edge node [right] {$ $} (yb2)
(ya1) edge node [right] {$ $} (yb3)
(yc3) edge node [right] {$ $} (yc1)
(yc3) edge node [right] {$ $} (yc2)
(yc3) edge node [right] {$ $} (ya1)
(yc3) edge node [right] {$ $} (ya2);
\node (zc1) at (5.5,1) {$c_1$};
\node (zc2) at (5.5,0) {$c_2$};
\node (zc3) at (5,.8) {$c_3$};
\node (zc4) at (5,.2) {$c_4$};
\node (za1) at (6,.8) {$a_1$};
\node (za2) at (6,.2) {$a_2$};
\node (zb3) at (6.5,.5) {$b_3$};
\node (zb1) at (7,1) {$b_1$};
\node (zb2) at (7,0) {$b_2$};
\path[font=\small,>=angle 90]
(zc1) edge node [right] {$ $} (zc2)
(za1) edge node [right] {$ $} (za2)
(zb1) edge node [right] {$ $} (zb2)
(zb1) edge node [right] {$ $} (zb3)
(zb2) edge node [right] {$ $} (zb3)
(zc1) edge node [right] {$ $} (za1)
(zc1) edge node [right] {$ $} (za2)
(zc2) edge node [right] {$ $} (za1)
(zc2) edge node [right] {$ $} (za2)
(za1) edge node [right] {$ $} (zb1)
(za2) edge node [right] {$ $} (zb2)
(za1) edge node [right] {$ $} (zb3)
(zc3) edge node [right] {$ $} (zc1)
(zc3) edge node [right] {$ $} (zc2)
(zc3) edge node [right] {$ $} (za1)
(zc3) edge node [right] {$ $} (za2)
(zc4) edge node [right] {$ $} (zc1)
(zc4) edge node [right] {$ $} (zc2)
(zc4) edge node [right] {$ $} (zc3)
(zc4) edge node [right] {$ $} (za1)
(zc4) edge node [right] {$ $} (za2);
\end{tikzpicture}
$
    \caption{The graphs $\SLi{2}{1}{2}$, $\SLi{2}{1}{3}$, and $\SLi{2}{1}{4}$}
    \label{fig21mL}
\end{figure}

\begin{lemma}\label{lem21ad}
Suppose that we are under the conditions of Hypothesis \ref{hyp21mL}. Further suppose that $\SLi{2}{1}{(t+1)}=\Delta(G)$ for some finite solvable group $G$, where $|G|$ is minimal. Then $G$ does not have any normal nonabelian Sylow $p$-subgroup for any $p\in\{a_1,a_2\}\cup\{b_1,b_3\}\cup\{c_1,c_2,\ldots,c_{t+1}\}$. In particular, $p$ is a strongly admissible vertex.
\end{lemma}
\begin{proof}
First we consider the vertex $a_1$. Notice that the graph $\SLi{2}{1}{(t+1)}[a_1]$ is of diameter three, which will violate Sass's main result from \cite{S}. Specifically, one can follow the conventional notation with $p=b_1$ and $q=c_1$, which yields $|\rho_3|=1<3$, contradicting Theorem \ref{diameter}\eqref{diam1}, for instance. Next, removing the edge between $a_1$ and $c_i$ where $1\leq i\leq t+1$ results in a graph that does not occur by P\'alfy's condition from \cite{P} using vertices $a_1$, $c_i$, and $b_2$ (see Lemma \ref{PC}). Therefore, the edge $\epsilon(a_1,c_i)$ cannot be lost for any $1\leq i\leq t+1$. Removing the edge between $a_1$ and $a_2$ yields a graph that does not occur by the generalized P\'alfy's condition from \cite{A}. In particular, one can see the $5$-cycle in the complement graph using the vertices $a_1$, $a_2$, $b_3$, $c_1$, and $b_2$. Hence, the edge $\epsilon(a_1,a_2)$ cannot be lost. Removing the edge between $a_1$ and $b_1$ again results in a graph of diameter three, which does not occur by way of \cite{S}. To see this, one can use an argument similar to the above with $p=b_1$ and $q=c_1$, for example. Next, removing the edge $\epsilon(a_1,b_3)$ is symmetric to that of removing the edge $\epsilon(a_1,b_1)$, and therefore an identical argument can be applied. Finally, losing both of the edges results in the graph $\SLi{2}{1}{(t+1)}[\epsilon(a_1, b_1), \epsilon(a_1, b_3)]$, which can be shown not to occur by once again employing a diameter three argument. Hence, $a_1$ satisfies the first two conditions in Definition \ref{adstrong}, and therefore is admissible.

To see that $a_1$ is strongly admissible, we need only show that the third condition of Definition \ref{adstrong} is met. First, we consider removing $a_1$ and all incident edges, along with possibly the edges $\epsilon(b_1,b_3)$, $\epsilon(a_2,c_i)$ for all $1\leq i\leq t+1$, and $\epsilon(c_i,c_j)$ for all $1\leq i<j\leq t+1$. When the edge between $b_1$ and $b_3$ is removed, the resulting graph does not occur by P\'alfy's condition using vertices $b_1$, $b_3$, and $c_1$. Therefore, $\epsilon(b_1,b_3)$ cannot be lost. Next, removing the edge between $a_2$ and $c_i$ for any $1\leq i\leq t+1$ yields a graph that does not occur by P\'alfy's condition using vertices $a_2$, $b_1$, and $c_i$. Thus, $\epsilon(a_2,c_i)$ cannot be lost for all $1\leq i\leq t+1$. Finally, removing the edge between $c_i$ and $c_j$ for any $1\leq i<j\leq t+1$ gives us a graph that does not occur yet again by P\'alfy's condition with the trio $c_i$, $c_j$, and $b_1$, and so $\epsilon(c_i,c_j)$ cannot be lost. Thus, none of the edges between two adjacent vertices of $a_1$ can be removed, and therefore neither can any combination of them. Hence, $a_1$ is strongly admissible.

Next, let us consider the vertex $a_2$. Notice that the graph $\SLi{2}{1}{(t+1)}[a_2]$ is of diameter three, which will violate Sass's main result from \cite{S} by labeling $p=b_2$ and $q=c_1$, for example. This yields $|\rho_3|=1<3$, contradicting Theorem \ref{diameter}\eqref{diam1}. Next we consider losing edges incident to $a_2$. It was shown above that the edge $\epsilon(a_1,a_2)$ cannot be removed. Next, it is clear that $\epsilon(a_2,c_i)$ cannot be lost for all $1\leq i\leq t+1$, since doing so would violate P\'alfy's condition with $a_2$, $c_i$, and $b_1$. Finally, removing the edge between $a_2$ and $b_2$ once again gives us a graph of diameter three that does not occur due to the main result from \cite{S}. Thus, $a_2$ is admissible.

As for $a_1$ above, we need only show that $a_2$ satisfies the third condition of Definition \ref{adstrong} to prove that $a_2$ is strongly admissible. First, we consider the graph once $a_2$ and all incident edges have been removed, along with potentially $\epsilon(a_1,c_i)$ for all $1\leq i\leq t+1$ and $\epsilon(c_i,c_j)$ for all $1\leq i<j\leq t+1$. However, it was proven above that neither $\epsilon(a_1,c_i)$ could be lost for all $1\leq i\leq t+1$, nor could $\epsilon(c_i,c_j)$ be lost for all $1\leq i<j\leq t+1$, both due to P\'alfy's condition. Thus, $a_2$ is strongly admissible.

Now, let us consider vertex $b_1$. Observe that $\SLi{2}{1}{(t+1)}[b_1]$ is isomorphic to the graph $\Gamma_{t+3,2}$ and therefore does not occur by Theorem \ref{KT} because $t+3>2$ (see \cite{BL} or \cite{BLL2}). Let us next consider the removal of the edges incident to $b_1$. The edge $\epsilon(b_1,b_2)$ cannot be lost, because that would violate P\'alfy's condition with $b_1$, $b_2$, and $c_1$. Likewise, $\epsilon(b_1,b_3)$ cannot be removed either due to the trio $b_1$, $b_3$, and $c_1$. Finally, it was already established above that the edge $\epsilon(b_1,a_1)$ cannot be removed. Thus, $b_1$ is admissible.

It is clear that $b_1$ satisfies condition three of Definition \ref{adstrong}, as we will show. Let us consider the graph when $b_1$, all incident edges, and possibly $\epsilon(b_2,b_3)$ and $\epsilon(a_1,b_3)$ have been removed. The edge $\epsilon(b_2,b_3)$ cannot be removed due to P\'alfy's condition ($b_2$, $b_3$, and $c_1$), nor can $\epsilon(a_1,b_3)$ be removed due to the addendum of Theorem \ref{KT}. Thus, $b_1$ is strongly admissible. By symmetry, we also get that $b_3$ is strongly admissible.

Finally, we consider $c_1$. Note that $\SLi{2}{1}{(t+1)}[c_1]=\SLi{2}{1}{t}$, which does not occur by our inductive hypothesis (see Hypothesis \ref{hyp21mL}). Further, let us examine the loss of the edges incident to $c_1$. It was established above that $\epsilon(a_1,c_1)$ and $\epsilon(a_2,c_1)$ cannot be removed due to P\'alfy's condition. Likewise, the removal of the edge $\epsilon(c_1,c_i)$ for all $2\leq i\leq t+1$ again violates P\'alfy's condition (using $c_1$, $c_i$, and $b_1$, for example). Therefore, $c_1$ is admissible.

We will now show that $c_1$ is strongly admissible. Let us examine the graph once $c_1$ and all incident edges have been removed, along with potentially $\epsilon(a_1,a_2)$, $\epsilon(a_1,c_i)$, and $\epsilon(a_2,c_i)$ for all $2\leq i\leq t+1$, as well as $\epsilon(c_i,c_j)$ for all $2\leq i<j\leq t+1$. These edges were all treated above in a similar manner and cannot here be removed as shown by P\'alfy's condition or the generalized P\'alfy's condition from \cite{A}. Therefore, $c_1$ is strongly admissible. Due to a symmetrical argument, we also get that $c_i$ is strongly admissible for all $2\leq i\leq t+1$.

Since the above vertices are all strongly admissible, then by Lemma \ref{strong}, there is no normal nonabelian Sylow $p$-subgroup for any $p\in\{a_1,a_2\}\cup\{b_1,b_3\}\cup\{c_1,c_2,\ldots,c_{t+1}\}$.
\end{proof}

We next address the remaining vertex of the graph $\SLi{2}{1}{(t+1)}$, which is the prime $b_2$.

\begin{lemma}\label{lem21pi}
Suppose that we are under the conditions of Hypothesis \ref{hyp21mL}. Further suppose that $\SLi{2}{1}{(t+1)}=\Delta(G)$ for some finite solvable group $G$, where $|G|$ is minimal. Then $G$ does not have any normal nonabelian Sylow $b_2$-subgroup.
\end{lemma}
\begin{proof}
Following the notation of Lemma \ref{pi}, we consider $q=b_2$, and consequently $\pi=\{a_2,b_1,b_3\}$ and $\rho=\{a_1,c_1,c_2,\ldots,c_{t+1}\}$. We then denote $\pi_1=\{a_2\}$ and $\pi_2=\{b_1,b_3\}$, which allows us to consider $v=b_1\in\pi_2$ and $s=a_1\in\rho$, where we know $s$ is admissible by Lemma \ref{lem21ad}. Finally we take $w=c_1\in\rho$, which we note is not adjacent to $v$. Hence, the result follows by Lemma \ref{pi}.
\end{proof}

\begin{lemma}\label{lem21mLcon}
Suppose that we are under the conditions of Hypothesis \ref{hyp21mL}. Then no proper connected subgraph of $\SLi{2}{1}{(t+1)}$ with the same vertex set occurs as the prime character degree graph of any solvable group.
\end{lemma}
\begin{proof}
In the course of verifying admissibility, we have already shown that none of the edges in the complete subgraphs on $\{a_1,a_2,c_1,c_2,\ldots,c_{t+1}\}$ and $\{b_1,b_2,b_3\}$ can be removed. This is due to P\'alfy's condition or its generalized version from \cite{A}. Therefore, we need only consider losing one or two of the edges $\epsilon(a_1,b_1)$, $\epsilon(a_1,b_3)$, and $\epsilon(a_2,b_2)$, since losing all three would result in a disconnected subgraph. As proven in Lemma \ref{lem21ad}, none of the three edges can be removed on its own. Losing exactly two of the edges listed above yields a proper connected subgraph of $\Gamma_{t+3,3}$ with the same vertex set, which does not occur by Theorem \ref{KT}. The result follows.
\end{proof}

Under the assumption that there exists some finite solvable group $G$ with $|G|$ minimal such that $\Delta(G)=\SLi{2}{1}{(t+1)}$, we now have the outcome that there are no normal nonabelian Sylow $p$-subgroups associated with $G$, for all vertices $p\in\rho(G)$. This is due to Lemmas \ref{lem21ad} and \ref{lem21pi}. Consequently, we can then turn our attention to the Fitting subgroup of $G$, for which we follow convention and denote by $F$. We then note that $\rho(G)=\pi(|G:F|)$, and therefore $\rho(G)=\rho(G/\Phi(G))$, where $\Phi(G)$ is the Frattini subgroup of $G$. We aim to show that $\Phi(G)$ is trivial. There are two cases to consider: either $\Delta(G/\Phi(G))$ is connected, or $\Delta(G/\Phi(G))$ is disconnected.

For the first case, supposing that $\Delta(G/\Phi(G))$ is connected, we see that since $G$ and $G/\Phi(G)$ have the same vertex set, then Lemma \ref{lem21mLcon} guarantees that their prime character degree graphs must be the same. Therefore, $\Delta(G)=\Delta(G/\Phi(G))$, and because $|G|$ is minimal, we get that $\Phi(G)=1$, as desired.

For the second case, supposing that $\Delta(G/\Phi(G))$ is disconnected, we see that the complete components must have the vertex sets $\{a_1,a_2,c_1,\ldots,c_{t+1}\}$ and $\{b_1,b_2,b_3\}$. This may not explicitly violate P\'alfy's inequality from \cite{P2}, and therefore may occur. This will either satisfy Example 2.4 or Example 2.6 from \cite{L}. We can then follow the argument similar to that of Lemma 2.9 from \cite{BLL2}, thus reducing to the case that $\Phi(G)=1$, which was what we wanted.

Now that $\Phi(G)=1$, we can apply Lemma III 4.4 from \cite{H} to obtain the existence of some subgroup $H$ of $G$ such that $G=HF$ where $H\cap F$ is also trivial. Using this, we next aim to show that $F$ is minimal normal in $G$.

\begin{lemma}\label{lem21mLfit}
Suppose that we are under the conditions of Hypothesis \ref{hyp21mL}. Further suppose that $\SLi{2}{1}{(t+1)}=\Delta(G)$ for some finite solvable group $G$, where $|G|$ is minimal. Then the Fitting subgroup $\textit{F}$ of $G$ is minimal normal in $G$.
\end{lemma}
\begin{proof}
Following the aforementioned discussion, $F$ is the Fitting subgroup of $G$ and $H$ is defined as above. Furthermore, we set $E$ to be the Fitting subgroup of $H$. For the sake of contradiction, suppose that $F$ is not minimal normal in $G$. Consequently, we suppose that there exists some normal subgroup $N$ of $G$ such that $1<N<F$. We will show that no such $N$ can exist.

Notice that Theorem III 4.5 of \cite{H} guarantees the existence of another normal subgroup $M$ of $G$ such that $F=N\times M$. Since we know $N<F$ and $F=N\times M$, this then means that $M>1$. These subgroups $N$ and $M$ are now both nontrivial, and so $\rho(G/N)\subset\rho(G)$ and $\rho(G/M)\subset\rho(G)$, where we note that those are proper subsets. In particular, since we now have that $\rho(G)\setminus\rho(G/N)$ is nonempty, then for any prime $q\in\rho(G)\setminus\rho(G/N)$, the group $G/N$ will have a normal nonabelian Sylow $q$-subgroup whose class is a formation. This forces $q\in\rho(G/M)$. The fallout from this is that the vertex set corresponding to the prime character degree graph for $G$ must be the union of the vertex sets for those corresponding to $G/N$ and $G/M$. To wit, $\rho(G)=\rho(G/N)\cup\rho(G/M)$. Following the argument outlined in \cite{LMS}, for instance, we can then conclude that $\rho(G)\setminus[\rho(G/N)\cap\rho(G/M)]$ must lie in a complete subgraph of $\Delta(G)$, and so $\rho(G)\setminus[\rho(G/N)\cap\rho(G/M)]$ must lie in one of the following sets: $A=\{b_1,b_2,b_3\}$, $B=\{a_1,b_1,b_3\}$, $C=\{a_2,b_2\}$, or $D=\{a_1,a_2,c_1,c_2,\ldots,c_{t+1}\}$.

First suppose that $\rho(G)\setminus[\rho(G/N)\cap\rho(G/M)]\subseteq A$. One can follow the argument given in \cite{L3}, for example, and conclude that $E$ is in fact the Hall $A$-subgroup of $H$, and so we can find a character $\chi\in\Irr(G)$ such that $b_1b_2b_3\mid\chi(1)$. Denoting $\theta$ as an irreducible constituent of $\chi_{FE}$, one can verify that $\chi(1)/\theta(1)$ divides $|G:FE|$ and that $\chi(1)$ is relatively prime to $|G:FE|$, which yields $\chi_{FE}=\theta$. Since $b_1$, $b_2$, and $b_3$ all divide $\theta(1)$, and the only possible primes that could divide a character in $\cd(G/FE)$ are the vertices in $\rho(G)\setminus A=\{a_1,a_2,c_1,c_2,\ldots,c_{t+1}\}$, then we get that $\cd(G/FE)$ is trivial due to Gallagher's Theorem, and so $G/FE$ is abelian. This will then force $O^{a_1}(G)<G$, for instance, which is a contradiction, since we know that $a_1$ is admissible by Lemma \ref{lem21ad}, in which case $O^{a_1}(G)=G$ by Lemma \ref{bigoh}. Hence, this case cannot occur.

Next we suppose the case where $\rho(G)\setminus[\rho(G/N)\cap\rho(G/M)]\subseteq B$, which will then yield
$$
\{a_2,b_2,c_1,c_2,\ldots,c_{t+1}\}\subseteq\rho(G/N)\cap\rho(G/M).
$$
Since $b_1\in\rho(G)=\rho(G/N)\cup\rho(G/M)$, then we can, without loss of generality, suppose that $b_1\in\rho(G/N)$. Since we also know that $\rho(G/N)$ is proper in $\rho(G)$, we then have the following three options for the vertex set $\rho(G/N)$: (i) $\rho(G/N)=\{a_2,b_1,b_2,c_1,c_2,\ldots,c_{t+1}\}$, (ii) $\rho(G/N)=\{a_1,a_2,b_1,b_2,c_1,c_2,\ldots,c_{t+1}\}$, or (iii) $\rho(G/N)=\{a_2,b_1,b_2,b_3,c_1,c_2,\ldots,c_{t+1}\}$.

Under case (i), one can see that no connected subgraph with this vertex could occur, because it has two cut vertices (see \cite{MQ}). Therefore, the only graph that could occur is the disconnected subgraph with complete components on the vertex sets $\{a_2,c_1,c_2,\ldots,c_{t+1}\}$ and $\{b_1,b_2\}$. Following \cite{L}, we then know that $G/N$ has a central Sylow $a_1$-subgroup or a central Sylow $b_3$-subgroup, which implies either $O^{a_1}(G)<G$ or $O^{b_3}(G)<G$. This is a contradiction since both $a_1$ and $b_3$ were shown to be admissible, and so $O^{a_1}(G)=G$ and $O^{b_3}(G)=G$, again by Lemma \ref{bigoh}.

Next, suppose (ii), in which case we either end up with a connected proper subgraph of $\Gamma_{t+3,2}$, which does not occur by the addendum of Theorem \ref{KT}, or we have a disconnected subgraph resulting in a central Sylow $b_3$ subgroup of $G/N$, which cannot happen as $O^{b_3}(G)=G$ as above.

Finally for (iii), it is not hard to see that the only subgraph that could occur is disconnected (whereby the graphs that are connected would have two cut vertices, which cannot happen again by \cite{MQ}), resulting in $G/N$ having a central Sylow $a_1$-subgroup, which cannot happen as $a_1$ is admissible. Thus, this case also cannot occur.

We next consider the case where we suppose $\rho(G)\setminus[\rho(G/N)\cap\rho(G/M)]\subseteq C$. Notice that this implies that
$$
\{a_1,b_1,b_3,c_1,c_2,\ldots,c_{t+1}\}\subseteq\rho(G/N)\cap\rho(G/M).
$$
Since $\rho(G/N)$ and $\rho(G/M)$ are both proper in $G$, and because $\rho(G)=\rho(G/N)\cup\rho(G/M)$, then without loss of generality, we must have that
$$
\rho(G/N)=\{a_1,a_2,b_1,b_3,c_1,c_2,\ldots,c_{t+1}\}
$$
and
$$
\rho(G/M)=\{a_1,b_1,b_2,b_3,c_1,c_2,\ldots,c_{t+1}\}.
$$
However, one can check that no connected subgraph with the vertex set $\rho(G/M)$ occurs, and therefore the disconnected subgraph with complete components on the vertex sets $\{a_1,c_1,c_2,\ldots,c_{t+1}\}$ and $\{b_1,b_2,b_3\}$ must occur. Again following the results from \cite{L}, this then yields a central Sylow $a_2$-subgroup of $G/M$, which implies $O^{a_2}(G)<G$. This cannot happen, as $a_2$ is admissible by way of Lemma \ref{lem21ad}. Hence, this case also cannot happen.

Finally, we must therefore suppose that $\rho(G)\setminus[\rho(G/N)\cap\rho(G/M)]\subseteq D$. In this case, notice that this forces $\{b_1,b_2,b_3\}$ to be in both $\rho(G/N)$ and $\rho(G/M)$. Without loss of generality, we may also suppose $c_1\in\rho(G/N)$. Before we look at the cases for what $\rho(G/N)$ can be, we must first address the symmetry of all the vertices $c_i$ for all $1\leq i\leq t+1$. Since these vertices all behave the same, we can arrange them in order, which will help reduce our subsequent cases. We therefore fix the notation of letting $C_i^*:=\{c_2,c_3,\ldots,c_i\}$ for all $2\leq i\leq t+1$. With this in hand, we have eight cases for $\rho(G/N)$: (i) $\{b_1,b_2,b_3,c_1\}$, (ii) $\{a_1,b_1,b_2,b_3,c_1\}$, (iii) $\{a_2,b_1,b_2,b_3,c_1\}$, (iv) $\{a_1,a_2,b_1,b_2,b_3,c_1\}$, (v) $\{b_1,b_2,b_3,c_1\}\cup C_i^*$, (vi) $\{a_1,b_1,b_2,b_3,c_1\}\cup C_i^*$, (vii) $\{a_2,b_1,b_2,b_3,c_1\}\cup C_i^*$, or (viii) $\{a_1,a_2,b_1,b_2,b_3,c_1\}\cup C_i^*$ for $i\neq t+1$. We note that taking $i=t+1$ in the final case would imply that $\rho(G/N)=\rho(G)$, which we know does not happen. Next, we address every case.

Supposing case (i), we have that $\rho(G/N)=\{b_1,b_2,b_3,c_1\}$, and the only graph that arises is the disconnected subgraph with complete components with vertex sets $\{b_1,b_2,b_3\}$ and $\{c_1\}$. Again using \cite{L}, we then have that $G/N$ has a central Sylow $p$-subgroup for some $p\in\{a_1,a_2,c_2,\ldots,c_{t+1}\}$, in which case $O^p(G)<G$. This is our contradiction, as Lemma \ref{bigoh} forces $O^p(G)=G$ for all such relevant $p$ because they were shown to be admissible in Lemma \ref{lem21ad}. Therefore, this case cannot occur.

Next, for (ii), the only graph that arises with this vertex set is disconnected with complete components with vertex sets $\{a_1,c_1\}$ and $\{b_1,b_2,b_3\}$. We get the same contradiction as above.

For both cases (iii) and (iv), no graph (connected or disconnected) occurs with these vertex sets.

Supposing case (v), we see that no such connected graph occurs, but the disconnected graph with complete components with vertex sets $\{b_1,b_2,b_3\}$ and $\{c_1\}\cup C_i^*$ do arise. This, however, is handled identically to how case (i) transpired. This is mainly due to the fact that every vertex outside of $\rho(G/N)$ is admissible.

Case (vi) falls just like case (ii). Likewise, case (vii) behaves like case (iii), and case (viii) is tamed just like case (iv).

Hence, none of these cases can occur, which means our original supposition must have been false, meaning that no such $N$ can occur. Thus, the Fitting subgroup $F$ must be minimal normal in $G$.
\end{proof}

\begin{proposition}\label{prop21mL}
The graph $\SLi{2}{1}{m}$ does not occur as the prime character degree graph of any solvable group for all $m\geq2$.
\end{proposition}
\begin{proof}
For the sake of contradiction, suppose that we are under the conditions of Hypothesis \ref{hyp21mL} and suppose $\SLi{2}{1}{(t+1)}=\Delta(G)$ for some finite solvable group $G$, where $|G|$ is minimal. Observe that, by Lemma \ref{lem21mLfit}, the Fitting subgroup $F$ of $G$ is minimal normal in $G$. Next, we follow the notation of Proposition \ref{final}, taking $a=a_2$, $b=b_3$, $c=c_1$, and $d=b_1$, where we note that $c$ and $d$ are admissible by Lemma \ref{lem21ad}. Our contradiction arises when we apply Proposition \ref{final}, and thus we have proven the nonoccurrence of $\SLi{2}{1}{(t+1)}$. This completes our inductive step, and therefore $\SLi{2}{1}{m}$ does not occur as the prime character degree graph of any solvable group for all $m\geq2$.
\end{proof}

We have now dealt with the case of $n=1$. Notice that Proposition \ref{prop21mL} classifies the graph $\SLi{2}{1}{3}$ as non-occurring, which is isomorphic to the graph $C_7$ in \cite{LS2}. Therefore this graph $C_7$ is officially classified. Now we turn our attention to the case of $n=2$, and one can see Figure \ref{fig22mL} for examples of such graphs.

\begin{figure}[htb]
    \centering
$
\begin{tikzpicture}[scale=2]
\node (xc1) at (0,1) {$c_1$};
\node (xc2) at (0,0) {$c_2$};
\node (xa1) at (.5,.8) {$a_1$};
\node (xa2) at (.5,.2) {$a_2$};
\node (xb3) at (1,.7) {$b_3$};
\node (xb4) at (1,.3) {$b_4$};
\node (xb1) at (1.5,1) {$b_1$};
\node (xb2) at (1.5,0) {$b_2$};
\path[font=\small,>=angle 90]
(xc1) edge node [right] {$ $} (xc2)
(xa1) edge node [right] {$ $} (xa2)
(xb1) edge node [right] {$ $} (xb2)
(xb1) edge node [right] {$ $} (xb3)
(xb1) edge node [right] {$ $} (xb4)
(xb2) edge node [right] {$ $} (xb3)
(xb2) edge node [right] {$ $} (xb4)
(xb3) edge node [right] {$ $} (xb4)
(xc1) edge node [right] {$ $} (xa1)
(xc1) edge node [right] {$ $} (xa2)
(xc2) edge node [right] {$ $} (xa1)
(xc2) edge node [right] {$ $} (xa2)
(xa1) edge node [right] {$ $} (xb1)
(xa2) edge node [right] {$ $} (xb2)
(xa1) edge node [right] {$ $} (xb3)
(xa2) edge node [right] {$ $} (xb4);
\node (yc1) at (2.75,1) {$c_1$};
\node (yc2) at (2.75,0) {$c_2$};
\node (yc3) at (2.25,.5) {$c_3$};
\node (ya1) at (3.25,.8) {$a_1$};
\node (ya2) at (3.25,.2) {$a_2$};
\node (yb3) at (3.75,.7) {$b_3$};
\node (yb4) at (3.75,.3) {$b_4$};
\node (yb1) at (4.25,1) {$b_1$};
\node (yb2) at (4.25,0) {$b_2$};
\path[font=\small,>=angle 90]
(yc1) edge node [right] {$ $} (yc2)
(ya1) edge node [right] {$ $} (ya2)
(yb1) edge node [right] {$ $} (yb2)
(yb1) edge node [right] {$ $} (yb3)
(yb1) edge node [right] {$ $} (yb4)
(yb2) edge node [right] {$ $} (yb3)
(yb2) edge node [right] {$ $} (yb4)
(yb3) edge node [right] {$ $} (yb4)
(yc1) edge node [right] {$ $} (ya1)
(yc1) edge node [right] {$ $} (ya2)
(yc2) edge node [right] {$ $} (ya1)
(yc2) edge node [right] {$ $} (ya2)
(ya1) edge node [right] {$ $} (yb1)
(ya2) edge node [right] {$ $} (yb2)
(ya1) edge node [right] {$ $} (yb3)
(ya2) edge node [right] {$ $} (yb4)
(yc3) edge node [right] {$ $} (yc1)
(yc3) edge node [right] {$ $} (yc2)
(yc3) edge node [right] {$ $} (ya1)
(yc3) edge node [right] {$ $} (ya2);
\node (zc1) at (5.5,1) {$c_1$};
\node (zc2) at (5.5,0) {$c_2$};
\node (zc3) at (5,.8) {$c_3$};
\node (zc4) at (5,.2) {$c_4$};
\node (za1) at (6,.8) {$a_1$};
\node (za2) at (6,.2) {$a_2$};
\node (zb3) at (6.5,.7) {$b_3$};
\node (zb4) at (6.5,.3) {$b_4$};
\node (zb1) at (7,1) {$b_1$};
\node (zb2) at (7,0) {$b_2$};
\path[font=\small,>=angle 90]
(zc1) edge node [right] {$ $} (zc2)
(za1) edge node [right] {$ $} (za2)
(zb1) edge node [right] {$ $} (zb2)
(zb1) edge node [right] {$ $} (zb3)
(zb1) edge node [right] {$ $} (zb4)
(zb2) edge node [right] {$ $} (zb3)
(zb2) edge node [right] {$ $} (zb4)
(zb3) edge node [right] {$ $} (zb4)
(zc1) edge node [right] {$ $} (za1)
(zc1) edge node [right] {$ $} (za2)
(zc2) edge node [right] {$ $} (za1)
(zc2) edge node [right] {$ $} (za2)
(za1) edge node [right] {$ $} (zb1)
(za2) edge node [right] {$ $} (zb2)
(za1) edge node [right] {$ $} (zb3)
(za2) edge node [right] {$ $} (zb4)
(zc3) edge node [right] {$ $} (zc1)
(zc3) edge node [right] {$ $} (zc2)
(zc3) edge node [right] {$ $} (za1)
(zc3) edge node [right] {$ $} (za2)
(zc4) edge node [right] {$ $} (zc1)
(zc4) edge node [right] {$ $} (zc2)
(zc4) edge node [right] {$ $} (zc3)
(zc4) edge node [right] {$ $} (za1)
(zc4) edge node [right] {$ $} (za2);
\end{tikzpicture}
$
    \caption{The graphs $\SLi{2}{2}{2}$, $\SLi{2}{2}{3}$, and $\SLi{2}{2}{4}$}
    \label{fig22mL}
\end{figure}
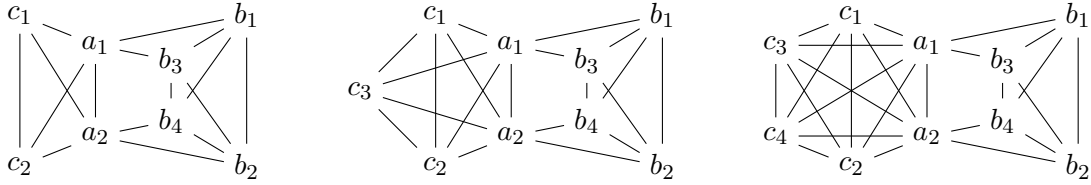

\begin{corollary}\label{cor22mL}
The graph $\SLi{2}{2}{m}$ does not occur as the prime character degree graph of any solvable group for all $m\geq2$.
\end{corollary}
\begin{proof}
Let $m\in\mathbb{N}$ such that $m\geq2$. We proceed by induction on $m$. Notice that for the base case of $m=2$, the graph $\SLi{2}{2}{2}$ was shown not to occur as the prime character degree graph of any solvable group in \cite{LS2}. In that paper, following the label given by Lewis and Summers, the graph $\SLi{2}{2}{2}$ is exactly the graph $D_{19}$.

For the inductive hypothesis, given any integer $t\geq2$, we assume that the graph $\SLi{2}{2}{t}$ does not occur as the prime character degree graph of any solvable group. Therefore, for the inductive step, we aim to show that the graph $\SLi{2}{2}{(t+1)}$ also does not occur as the prime character degree graph of any solvable group. Following the argument in \cite{LS2}, we will show that every vertex is admissible.

We begin with showing that the vertex $a_1$ is admissible. By removing $a_1$ and all incident edges, we are left with a graph of diameter three. This graph $\SLi{2}{2}{(t+1)}[a_1]$ violates the main result from \cite{S} (see Theorem \ref{diameter}). Removal of the edge between $a_1$ and $a_2$ from the graph $\SLi{2}{2}{(t+1)}$ yields a five-cycle in its complement, and therefore violates the generalized P\'alfy's condition from \cite{A}. Therefore, the edge $\epsilon(a_1,a_2)$ cannot be lost. Next, losing any edge $\epsilon(a_1,c_i)$ for $1\leq i\leq t+1$ violates P\'alfy's condition with $a_1$, $b_2$, and $c_i$, and hence, none of those edges can be lost as well. Finally, the vertices $b_1$ and $b_3$ are symmetric, and therefore the edges $\epsilon(a_1,b_1)$ and $\epsilon(a_1,b_3)$ are symmetric. Losing either or both of the aforementioned edges results in a graph of diameter three which will violate the main result from \cite{S}, and hence neither of those edges can be lost. We now have that $a_1$ satisfies the first two conditions from Definition \ref{adstrong}, and therefore is admissible. By symmetry, $a_2$ is also admissible.

Next, we consider the vertex $b_1$. Removing $b_1$ and all incident edges results in the graph $\SLi{2}{2}{(t+1)}[b_1]=\SLi{2}{1}{(t+1)}$, which was shown not to occur in Proposition \ref{prop21mL}. Removal of the edge between $b_1$ and $a_1$ from the graph $\SLi{2}{2}{(t+1)}$ was already considered above, and was shown it cannot be lost. Moreover, removal of any edge $\epsilon(b_1,b_j)$ for $2\leq j\leq4$ results in a violation of P\'alfy's condition with the trio $b_1$, $b_j$, and $c_1$, and therefore none of those edges can be lost either. Thus, $b_1$ is an admissible vertex, and by a symmetric argument, so are $b_2$, $b_3$, and $b_4$.

Finally, we consider the vertex $c_1$. Removing $c_1$ and all incident edges results in the graph $\SLi{2}{2}{(t+1)}[c_1]=\SLi{2}{2}{t}$, which does not occur as the prime character degree graph of any solvable group by the inductive hypothesis. The removal of the edge $\epsilon(a_1,c_1)$ was already considered above, and it was shown it cannot be lost. Likewise for $\epsilon(a_2,c_1)$. Finally, we see that removing any edge $\epsilon(c_1,c_i)$ for $2\leq i\leq t+1$ results in a violation of P\'alfy's condition with $c_1$, $c_i$, and $b_1$, and hence none of those edges can be lost. Thus, $c_1$ is an admissible vertex. Again by symmetry, we see that $c_i$ is also admissible for all $2\leq i\leq t+1$.

Since every vertex in the graph $\SLi{2}{2}{(t+1)}$ is admissible, then we know $\SLi{2}{2}{(t+1)}$ does not occur as the prime character degree graph of any solvable group by Theorem \ref{admissible}. Thus, by induction, the graph $\SLi{2}{2}{m}$ does not occur as the prime character degree graph of any solvable group for all $m\geq2$.
\end{proof}

\subsection{The case of $k\geq3$}

The main reason for handling the case of $k=2$ separately is that it can be viewed as the base case of part of an induction argument. In particular, our first goal here is to show that the graph $\SLi{k}{1}{m}$ does not occur for all $k\geq2$ and for all $m\geq1$.

Inducting first on $k$, we can consider the base case of $k=2$, which yields the graphs $\SLi{2}{1}{m}$. These graphs are shown not to occur for all $m\geq1$ by Proposition \ref{prop21mL}, noting that the case of $m=1$ yields a graph with six vertices classified as non-occurring in \cite{BLL}. Next, as our inductive hypothesis, given any integer $v\geq2$, we assume that the graph $\SLi{v}{1}{m}$ does not occur for all $m\geq1$. Then we would aim to show that the graph $\SLi{v+1}{1}{m}$ does not occur as the prime character degree graph of any solvable group for all $m\geq1$.

In order to show this nonoccurrence, we would next need to induct on $m$. For the base case of $m=1$, notice that the graph $\SLi{v+1}{1}{1}=\SL{v+1}{1}$ was shown not to occur by the main theorem in \cite{LM} (see Theorem \ref{Lfam}). Furthermore, for this inductive step, given any integer $t\geq1$, we assume that the graph $\SLi{v+1}{1}{t}$ does not occur as the prime character degree graph of any solvable group, and then we would aim to show the nonoccurrence of $\SLi{v+1}{1}{(t+1)}$. Before jumping into the proof of this, we organize our assumptions as follows:

\begin{hyp}\label{hypk1mL}
Given any integers $v\geq2$ and $t\geq1$, we assume the following:
\begin{enumerate}[(i)]
    \item\label{hypk1i} the graph $\SLi{v}{1}{m}$ does not occur for all $m\geq1$, and
    \item\label{hypk1ii} the graph $\SLi{v+1}{1}{t}$ does not occur.
\end{enumerate}
\end{hyp}

One can see Figure \ref{figk1mL} for examples of graphs in this family, which correspond to those in Proposition \ref{propk1mL}.

\begin{figure}[htb]
    \centering
$
\begin{tikzpicture}[scale=2]
\node (wc1) at (0,0.8) {$c_1$};
\node (wc2) at (0,0.2) {$c_2$};
\node (wa2) at (0.5,1) {$a_2$};
\node (wa3) at (0.5,0) {$a_3$};
\node (wa1) at (1,0.5) {$a_1$};
\node (wb1) at (1.5,0.75) {$b_1$};
\node (wb4) at (1.5,0.25) {$b_4$};
\node (wb2) at (2,1) {$b_2$};
\node (wb3) at (2,0) {$b_3$};
\path[font=\small,>=angle 90]
(wc1) edge node [right] {$ $} (wc2)
(wc1) edge node [right] {$ $} (wa2)
(wc1) edge node [right] {$ $} (wa3)
(wc1) edge node [right] {$ $} (wa1)
(wc2) edge node [right] {$ $} (wa2)
(wc2) edge node [right] {$ $} (wa3)
(wc2) edge node [right] {$ $} (wa1)
(wa2) edge node [right] {$ $} (wa3)
(wa2) edge node [right] {$ $} (wa1)
(wa3) edge node [right] {$ $} (wa1)
(wa3) edge node [right] {$ $} (wb3)
(wb1) edge node [right] {$ $} (wb2)
(wb1) edge node [right] {$ $} (wb4)
(wb1) edge node [right] {$ $} (wb3)
(wb2) edge node [right] {$ $} (wb4)
(wb2) edge node [right] {$ $} (wb3)
(wb4) edge node [right] {$ $} (wb3)
(wa1) edge node [right] {$ $} (wb1)
(wa1) edge node [right] {$ $} (wb4)
(wa2) edge node [right] {$ $} (wb2);
\node (xa1) at (4.25,0.5) {$a_1$};
\node (xa2) at (3.75,1) {$a_2$};
\node (xa3) at (3.75,0) {$a_3$};
\node (xc1) at (3.25,1) {$c_1$};
\node (xc2) at (3.25,0) {$c_2$};
\node (xc3) at (2.75,0.5) {$c_3$};
\node (xb1) at (4.75,0.75) {$b_1$};
\node (xb4) at (4.75,0.25) {$b_4$};
\node (xb2) at (5.25,1) {$b_2$};
\node (xb3) at (5.25,0) {$b_3$};
\path[font=\small,>=angle 90]
(xc1) edge node [right] {$ $} (xc2)
(xc1) edge node [right] {$ $} (xa2)
(xc1) edge node [right] {$ $} (xa3)
(xc1) edge node [right] {$ $} (xa1)
(xc2) edge node [right] {$ $} (xa2)
(xc2) edge node [right] {$ $} (xa3)
(xc2) edge node [right] {$ $} (xa1)
(xa2) edge node [right] {$ $} (xa3)
(xa2) edge node [right] {$ $} (xa1)
(xa3) edge node [right] {$ $} (xa1)
(xb1) edge node [right] {$ $} (xb2)
(xb1) edge node [right] {$ $} (xb4)
(xb1) edge node [right] {$ $} (xb3)
(xb2) edge node [right] {$ $} (xb4)
(xb2) edge node [right] {$ $} (xb3)
(xb4) edge node [right] {$ $} (xb3)
(xc1) edge node [right] {$ $} (xc3)
(xc3) edge node [right] {$ $} (xc2)
(xc3) edge node [right] {$ $} (xa1)
(xa2) edge node [right] {$ $} (xc3)
(xa3) edge node [right] {$ $} (xc3)
(xa1) edge node [right] {$ $} (xb1)
(xa1) edge node [right] {$ $} (xb4)
(xa2) edge node [right] {$ $} (xb2)
(xa3) edge node [right] {$ $} (xb3);
\node (ya1) at (0.5,-.7) {$a_1$};
\node (ya2) at (0.5,-1.3) {$a_2$};
\node (ya3) at (0,-.5) {$a_3$};
\node (ya4) at (0,-1.5) {$a_4$};
\node (yc1) at (-.5,-0.7) {$c_1$};
\node (yc2) at (-.5,-1.3) {$c_2$};
\node (yb1) at (2,-0.7) {$b_1$};
\node (yb2) at (2,-1.3) {$b_2$};
\node (yb3) at (1.5,-.5) {$b_3$};
\node (yb4) at (1.5,-1.5) {$b_4$};
\node (yb5) at (1,-1) {$b_5$};
\path[font=\small,>=angle 90]
(yc1) edge node [right] {$ $} (yc2)
(yc1) edge node [right] {$ $} (ya1)
(yc1) edge node [right] {$ $} (ya2)
(yc1) edge node [right] {$ $} (ya3)
(yc2) edge node [right] {$ $} (ya2)
(yc2) edge node [right] {$ $} (ya3)
(yc2) edge node [right] {$ $} (ya1)
(ya2) edge node [right] {$ $} (ya3)
(ya2) edge node [right] {$ $} (ya1)
(ya3) edge node [right] {$ $} (ya1)
(yb1) edge node [right] {$ $} (yb2)
(yb1) edge node [right] {$ $} (yb4)
(yb1) edge node [right] {$ $} (yb3)
(yb2) edge node [right] {$ $} (yb4)
(yb2) edge node [right] {$ $} (yb3)
(yb4) edge node [right] {$ $} (yb3)
(ya1) edge node [right] {$ $} (yb1)
(ya2) edge node [right] {$ $} (yb2)
(ya3) edge node [right] {$ $} (yb3)
(ya3) edge node [right] {$ $} (yb3)
(ya4) edge node [right] {$ $} (yc2)
(ya4) edge node [right] {$ $} (yc1)
(ya4) edge node [right] {$ $} (ya3)
(ya4) edge node [right] {$ $} (ya1)
(ya4) edge node [right] {$ $} (ya2)
(ya1) edge node [right] {$ $} (yb5)
(yb5) edge node [right] {$ $} (yb1)
(yb5) edge node [right] {$ $} (yb2)
(yb3) edge node [right] {$ $} (yb5)
(yb4) edge node [right] {$ $} (yb5)
(ya4) edge node [right] {$ $} (yb4);
\node (za1) at (4.25,-0.7) {$a_1$};
\node (za2) at (4.25,-1.3) {$a_2$};
\node (za3) at (3.75,-0.5) {$a_3$};
\node (za4) at (3.75,-1.5) {$a_4$};
\node (zc1) at (3.25,-0.6) {$c_1$};
\node (zc2) at (3.25,-1.4) {$c_2$};
\node (zc3) at (2.75,-1) {$c_3$};
\node (zb1) at (5.75,-0.7) {$b_1$};
\node (zb2) at (5.75,-1.3) {$b_2$};
\node (zb3) at (5.25,-0.5) {$b_3$};
\node (zb4) at (5.25,-1.5) {$b_4$};
\node (zb5) at (4.75,-1) {$b_5$};
\path[font=\small,>=angle 90]
(zb3) edge node [right] {$ $} (zb5)
(zb4) edge node [right] {$ $} (zb5)
(zc1) edge node [right] {$ $} (zc2)
(zc1) edge node [right] {$ $} (za2)
(zc1) edge node [right] {$ $} (za3)
(zc1) edge node [right] {$ $} (za1)
(zc2) edge node [right] {$ $} (za2)
(zc2) edge node [right] {$ $} (za3)
(zc2) edge node [right] {$ $} (za1)
(za2) edge node [right] {$ $} (za3)
(za2) edge node [right] {$ $} (za1)
(za3) edge node [right] {$ $} (za1)
(zb1) edge node [right] {$ $} (zb2)
(zb1) edge node [right] {$ $} (zb4)
(zb1) edge node [right] {$ $} (zb3)
(zb2) edge node [right] {$ $} (zb4)
(zb2) edge node [right] {$ $} (zb3)
(zb4) edge node [right] {$ $} (zb3)
(zc1) edge node [right] {$ $} (zc3)
(zc3) edge node [right] {$ $} (zc2)
(zc3) edge node [right] {$ $} (za1)
(za2) edge node [right] {$ $} (zc3)
(za3) edge node [right] {$ $} (zc3)
(za1) edge node [right] {$ $} (zb1)
(za2) edge node [right] {$ $} (zb2)
(za3) edge node [right] {$ $} (zb3)
(za3) edge node [right] {$ $} (zb3)
(za4) edge node [right] {$ $} (zc2)
(za4) edge node [right] {$ $} (zc3)
(za4) edge node [right] {$ $} (zc1)
(za4) edge node [right] {$ $} (za3)
(za4) edge node [right] {$ $} (za1)
(za4) edge node [right] {$ $} (za2)
(za1) edge node [right] {$ $} (zb5)
(zb5) edge node [right] {$ $} (zb1)
(zb5) edge node [right] {$ $} (zb2)
(za4) edge node [right] {$ $} (zb4);
\end{tikzpicture}
$
    \caption{The graphs $\SLi{3}{1}{2}$, $\SLi{3}{1}{3}$, $\SLi{4}{1}{2}$, and $\SLi{4}{1}{3}$}
    \label{figk1mL}
\end{figure}
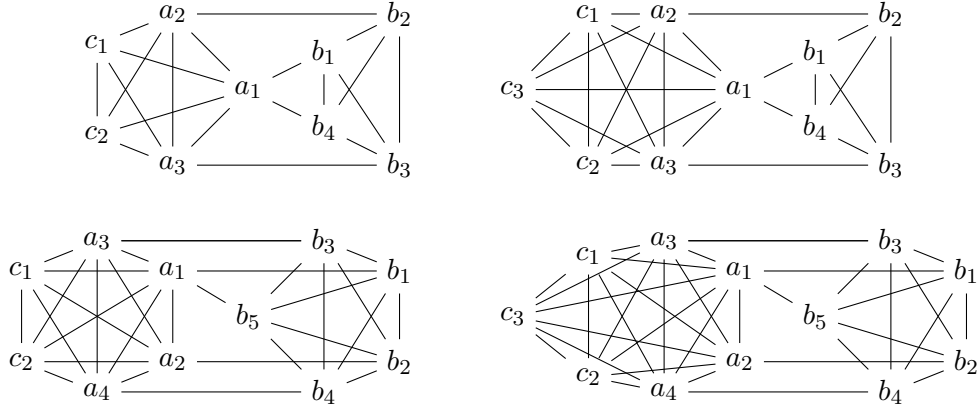

\begin{proposition}\label{propk1mL}
The graph $\SLi{k}{1}{m}$ does not occur as the prime character degree graph of any solvable group for all $k\geq3$ and $m\geq2$.
\end{proposition}
\begin{proof}
Suppose that we are under the conditions of Hypothesis \ref{hypk1mL}. To complete our induction argument on both $k$ and $m$, we need only show that the graph $\SLi{v+1}{1}{(t+1)}$ does not occur as the prime character degree graph of any solvable group. We will do this by showing that every vertex is admissible.

First we aim to show $a_1$ is admissible. We start by considering the graph determined by removing $a_1$ and all incident edges, which can be denoted by $\SLi{v+1}{1}{(t+1)}[a_1]$. Since $\SLi{v+1}{1}{(t+1)}[a_1]$ is a connected subgraph of $\Gamma_{v+t+1,v+2}$ with the same vertex set, then $\SLi{v+1}{1}{(t+1)}[a_1]$ does not occur by the addendum of Theorem \ref{KT}. Next we consider losing edges incident to $a_1$. Losing the edge $\epsilon(a_1,a_i)$ for any $2\leq i\leq v+1$ violates P\'alfy's condition with $a_1$, $a_i$, and $b_j$ for any $2\leq j\leq v+1$ with $i\neq j$. Therefore, none of those edges can be lost. Removing the edge $\epsilon(a_1,c_i)$ for any $1\leq i\leq t+1$ can again use P\'alfy's condition, this time with the trio $a_1$, $c_i$, and $b_2$, for example. Hence, none of those edges can be lost either. Finally, removing one or both of the edges $\epsilon(a_1,b_1)$ and $\epsilon(a_1,b_{v+2})$ yields a proper connected subgraph of $\Gamma_{v+t+2,v+2}$ with the same vertex set, and so again using Theorem \ref{KT}, we get nonoccurrence. Thus, $a_1$ is admissible.

We aim to show $a_2$ is admissible. First, we consider the resulting graph by removing $a_2$ and all incident edges, denoted by $\SLi{v+1}{1}{(t+1)}[a_2]$. For ease of notation, we will instead refer to this as $\Delta(H)$, where $H$ is some solvable subgroup of $G$. We partition the vertices of $\Delta(H)$ by way of Theorem \ref{diameter}, which forces $\rho_3 = \{a_1,a_3,\ldots,a_{v+1}\}$. By Theorem \ref{diameter}\eqref{diam4}, we have the existence of a normal Sylow $p_1$-subgroup for exactly one $p_1 \in \rho_3$. Without loss of generality, suppose $p_1 = a_1$; denote the associated normal subgroup as $P_1$. By \cite{L2.5}, we know that $\rho(H/P'_1) = \rho(H) \setminus \{p_1\}$. Observe that $\Delta(H/P'_1)$ is a connected subgraph of $\Delta(H)$ obtained by removing $a_1$ and all incident edges, and possibly edges between vertices adjacent to $a_1$. Since $\Delta(H/P'_1)$ is diameter three, we again partition its vertex set by Theorem \ref{diameter}, which yields its corresponding $\rho_3 = \{a_3,\ldots,a_{v+1}\}$. By Theorem \ref{diameter}\eqref{diam4}, we now have that $H/P'_1$ has a normal Sylow $p_2$-subgroup for exactly one prime $p_2 \in \rho_3$. Without loss of generality, suppose $p_2 = a_3$; denote the corresponding normal subgroup as $P_2$. It can then be shown that $P_1P_2 = P_1 \times P_2$ is normal in $H$. Thus, $P_2$ is a normal Sylow $p_2$-subgroup of $H$ where $p_2 \in \rho_3$. This implies that $P_1$ and $P_2$ are distinct normal Sylow $p$-subgroups of $H$ for more than one $p \in \rho_3$. This is a contradiction of Theorem \ref{diameter}\eqref{diam4}. Hence, again using the notation from above, we get that $\Delta(H)$ cannot occur. To finish the argument on the admissibility, we consider losing edges incident to $a_2$. Clearly, the edges $\epsilon(a_2,a_i)$ ($1 \leq i \leq v+1$ and $i\neq2$) and $\epsilon(a_2,c_j)$ ($1 \leq j\leq t+1$) cannot be lost by the same argument used in the case of $a_1$ above. Finally, concerning the edge between $a_2$ and $b_2$, we see that the resulting graph $\SLi{v+1}{1}{(t+1)}[\epsilon(a_2,b_2)]$ can be shown not to occur by an argument similar to that used for $\SLi{v+1}{1}{(t+1)}[a_2]$. Hence, $a_2$ is admissible. Consequently $a_i$ is admissible for all $3 \leq i \leq v+1$ by symmetry.

Next we aim to show that $b_1$ is admissible. Notice that by removing $b_1$ and all incident edges we are left with the graph $\Gamma_{v+t+2,v+1}$ which does not occur by Theorem \ref{KT}. Considering, now, removing edges incident to $b_1$, we see that the edge between $b_1$ and $b_i$ cannot be lost due to P\'alfy's condition with $b_1$, $b_i$, and $c_1$ for all $2\leq i\leq v+2$. Furthermore, the edge between $b_1$ and $a_1$ was considered above, and it cannot be lost. Hence, $b_1$ is admissible, and by symmetry, so is $b_{v+2}$.

For the vertex $b_2$, we start by considering $\SLi{v+1}{1}{(t+1)}[b_2]$, which is isomorphic to the graph $\SLi{v}{1}{(t+2)}$, which does not occur by Hypothesis \ref{hypk1mL}\eqref{hypk1i}. Next we consider removing edges incident to $b_2$. Clearly the edge $\epsilon(b_2,b_i)$ cannot be lost for any $1\leq i\leq v+2$ such that $i\neq2$ due to P\'alfy's condition with $b_2$, $b_i$, and $c_1$. Moreover, the edge between $b_2$ to $a_2$ was handled above, and cannot be lost. Hence, $b_2$ is admissible. By an identical argument, we also get that $b_i$ is admissible for all $3\leq i\leq v+1$.

Finally, we aim to show $c_1$ is admissible. First, the graph $\SLi{v+1}{1}{(t+1)}[c_1]$ is isomorphic to the graph $\SLi{v+1}{1}{t}$ which does not occur as the prime character degree graph of any solvable group by Hypothesis \ref{hypk1mL}\eqref{hypk1ii}. We now consider the edges incident to $c_1$. Losing the edge $\epsilon(c_1, c_i)$ for any $2 \leq i \leq t + 1$ violates P\'alfy's condition with $c_1$, $c_i$, and $b_1$. Also, as shown above, $\epsilon(a_i,c_1)$ cannot be lost for any $1\leq i\leq v+1$. Hence $c_1$ is admissible, and by symmetry, $c_i$ is admissible for all $2\leq i\leq t+1$.

Since every vertex is admissible, then $\SLi{v+1}{1}{(t+1)}$ does not occur by Theorem \ref{admissible}. Hence, by induction, $\SLi{k}{1}{m}$ does not occur as the prime character degree graph of any solvable group for all $k\geq3$ and $m\geq2$.
\end{proof}

Now that we have tamed the indices $k$ and $m$, we now turn our attention to the index $n$. In order to handle $n$, which ranges as $1\leq n\leq k$, we must yet again produce an induction argument, where we will see Proposition \ref{propk1mL} act as part of a base case.

Whenever $k\geq2$, $n\geq1$, and $m\geq1$, we now aim to do a triple induction to show that the graph $\SLi{k}{n}{m}$ does not occur as the prime character degree graph of any solvable group.

Our first stage, inducting on $k$, has as our base case the scenario of $k = 2$. The resulting graph $\SLi{2}{n}{m}$ does not occur for all $1\leq n \leq 2$ and for all $m \geq 1$. The case $n = 1$ was proven in Proposition \ref{prop21mL} and $n = 2$ was proven in Corollary \ref{cor22mL}. For our inductive hypothesis: given any integer $v\geq 2$, we assume that the graph $\SLi{v}{n}{m}$ does not occur for all $1\leq n \leq v$ and for all $m\geq 1$. For our inductive step in this stage, we will show that the graph $\SLi{v+1}{n}{m}$ does not occur for all $1 \leq n \leq v+1$ and for all $m \geq 1$. In order to show this, we next proceed by induction on $n$.

Notice that the base case of $n = 1$ requires the graph $\SLi{v+1}{1}{m}$ to not occur for all $m\geq 1$, which was proven in Proposition \ref{propk1mL}. For our inductive hypothesis, given any integer $1 \leq w < v+1$, we assume the graph $\SLi{v+1}{w}{m}$ does not occur for all $m \geq 1$. We will prove the inductive step, that the graph $\SLi{v+1}{w+1}{m}$ does not occur for all $m\geq 1$, by inducting on $m$. 

Here our final base case of $m = 1$ results in the graph $\SLi{v+1}{w+1}{1}$, which is classified as non-occurring by Theorem \ref{Lfam}. Our final inductive hypothesis is that given any integer $t\geq 1$ we assume the graph $\SLi{v+1}{w+1}{t}$ does not occur. We now consolidate our inductive hypotheses as:

\begin{hyp}\label{hypknmL}
Given any integers $v \geq 2$ and $1\leq w < v+1$ and $t \geq 1$, we assume the following:
\begin{enumerate}[(i)]
    \item\label{hypkni} $\SLi{v}{n}{m}$ does not occur for all $1 \leq n \leq v$ and for all $m\geq 1$,
    \item\label{hypknii} $\SLi{v+1}{w}{m}$ does not occur for all $m\geq 1$, and
    \item\label{hypkniii} $\SLi{v+1}{w+1}{t}$ does not occur.
\end{enumerate}
\end{hyp}

Our goal, therefore, in order to finish this triple inductive argument, will be to show that the graph $\SLi{v+1}{w+1}{(t+1)}$ does not occur as the prime character degree graph of any solvable group. We address this in the following result, and one can see examples of these graphs in Figure \ref{figknmL} below.

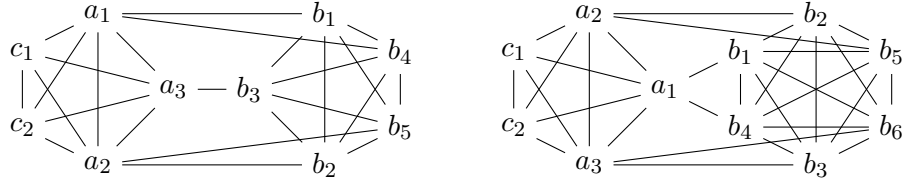
\begin{figure}[htb]
    \centering
$
\begin{tikzpicture}[scale=2]
\node (xc1) at (0,.75) {$c_1$};
\node (xc2) at (0,0.25) {$c_2$};
\node (xa1) at (0.5,1) {$a_1$};
\node (xa2) at (0.5,0) {$a_2$};
\node (xa3) at (1,0.5) {$a_3$};
\node (xb1) at (2,1) {$b_1$};
\node (xb4) at (2.5,.75) {$b_4$};
\node (xb2) at (2,0) {$b_2$};
\node (xb5) at (2.5,.25) {$b_5$};
\node (xb3) at (1.5,0.5) {$b_3$};
\path[font=\small,>=angle 90]
(xa1) edge node [right] {$ $} (xb1)
(xa1) edge node [right] {$ $} (xb4)
(xc1) edge node [right] {$ $} (xc2)
(xc1) edge node [right] {$ $} (xa1)
(xc1) edge node [right] {$ $} (xa2)
(xc1) edge node [right] {$ $} (xa3)
(xc2) edge node [right] {$ $} (xa1)
(xc2) edge node [right] {$ $} (xa2)
(xc2) edge node [right] {$ $} (xa3)
(xa3) edge node [right] {$ $} (xb3)
(xa1) edge node [right] {$ $} (xa3)
(xa2) edge node [right] {$ $} (xa3)
(xa1) edge node [right] {$ $} (xa2)
(xa2) edge node [right] {$ $} (xb2)
(xa2) edge node [right] {$ $} (xb5)
(xb1) edge node [right] {$ $} (xb4)
(xb1) edge node [right] {$ $} (xb2)
(xb4) edge node [right] {$ $} (xb2)
(xb4) edge node [right] {$ $} (xb5)
(xb1) edge node [right] {$ $} (xb5)
(xb2) edge node [right] {$ $} (xb5)
(xb1) edge node [right] {$ $} (xb3)
(xb4) edge node [right] {$ $} (xb3)
(xb2) edge node [right] {$ $} (xb3)
(xb5) edge node [right] {$ $} (xb3);
\node (yc1) at (3.25,0.75) {$c_1$};
\node (yc2) at (3.25,0.25) {$c_2$};
\node (ya1) at (4.25,.5) {$a_1$};
\node (ya2) at (3.75,1) {$a_2$};
\node (ya3) at (3.75,0) {$a_3$};
\node (yb1) at (4.75,0.75) {$b_1$};
\node (yb4) at (4.75,.25) {$b_4$};
\node (yb2) at (5.25,1) {$b_2$};
\node (yb5) at (5.75,.75) {$b_5$};
\node (yb3) at (5.25,0) {$b_3$};
\node (yb6) at (5.75,0.25) {$b_6$};
\path[font=\small,>=angle 90]
(ya1) edge node [right] {$ $} (yb1)
(ya1) edge node [right] {$ $} (yb4)
(yc1) edge node [right] {$ $} (yc2)
(yc1) edge node [right] {$ $} (ya1)
(yc1) edge node [right] {$ $} (ya2)
(yc1) edge node [right] {$ $} (ya3)
(yc2) edge node [right] {$ $} (ya1)
(yc2) edge node [right] {$ $} (ya2)
(yc2) edge node [right] {$ $} (ya3)
(ya3) edge node [right] {$ $} (yb3)
(ya1) edge node [right] {$ $} (ya3)
(ya2) edge node [right] {$ $} (ya3)
(ya1) edge node [right] {$ $} (ya2)
(ya2) edge node [right] {$ $} (yb2)
(ya2) edge node [right] {$ $} (yb5)
(yb1) edge node [right] {$ $} (yb4)
(yb1) edge node [right] {$ $} (yb2)
(yb4) edge node [right] {$ $} (yb2)
(yb4) edge node [right] {$ $} (yb5)
(yb1) edge node [right] {$ $} (yb5)
(yb2) edge node [right] {$ $} (yb5)
(yb1) edge node [right] {$ $} (yb3)
(yb4) edge node [right] {$ $} (yb3)
(yb2) edge node [right] {$ $} (yb3)
(yb3) edge node [right] {$ $} (yb5)
(ya3) edge node [right] {$ $} (yb6)
(yb1) edge node [right] {$ $} (yb6)
(yb2) edge node [right] {$ $} (yb6)
(yb3) edge node [right] {$ $} (yb6)
(yb4) edge node [right] {$ $} (yb6)
(yb5) edge node [right] {$ $} (yb6);
\end{tikzpicture}
$
    \caption{The graphs $\SLi{3}{2}{2}$ and $\SLi{3}{3}{2}$}
    \label{figknmL}
\end{figure}

\begin{corollary}\label{corknmL}
The graph $\SLi{k}{n}{m}$ does not occur as the prime character degree graph of any solvable group for all $k\geq3$ and $2\leq n\leq k$ and $m\geq2$.
\end{corollary}
\begin{proof}
Suppose that we are under the conditions of Hypothesis \ref{hypknmL}.
To complete our triple induction argument on $k$, $n$, and $m$, we will show that the graph $\SLi{v+1}{w+1}{(t+1)}$ does not occur.
To that end, we will proceed by showing every vertex is admissible.

First, for any vertex $a_i$ ($1\leq i\leq v+1$), we can employ the same methods as seen in previous results. Technically, the vertices $a_1,\ldots,a_{w+1}$ and the vertices $a_{w+2},\ldots,a_{v+1}$ behave slightly differently based on their edges, but the argument seen in Proposition \ref{propk1mL} works in the same way regardless. We omit those arguments here, but using them, one can conclude that $a_i$ is admissible for all $1\leq i\leq v+1$. 

In a similar fashion as above, the vertices in the sets $\beta=\{b_1,\ldots,b_{w+1},b_{v+2},\ldots,b_{v+w+2}\}$ and $\overline{\beta}=\{b_{w+2},\ldots,b_{v+1}\}$ also behave differently. For our purposes, it is then sufficient to consider $b_1$ and $b_{w+2}$.

Losing the vertex $b_1$ results in a graph isomorphic to $\SLi{v+1}{w}{(t+1)}$ which does not occur by Hypothesis \ref{hypknmL}\eqref{hypknii}. Again, as seen in earlier arguments, one can also show that no combination of edges incident to $b_1$ can be removed. Therefore, $b_1$ is admissible, and by symmetry we get that every vertex in $\beta$ is too.

Next, we consider the vertex $b_{w+2}$ (supposing that it exists, as long as $w+1<v+1$). The loss of $b_{w+2}$ results in the graph $\SLi{v}{w+1}{(t+2)}$ which does not occur by Hypothesis \ref{hypknmL}\eqref{hypkni}. The edges incident to $b_{w+2}$ follow identically to $b_1$. Hence, $b_{w+2}$ is admissible. Thus, by symmetry, we have that every vertex in $\overline{\beta}$ is admissible.

Obviously, the loss of $c_1$ and all incident edges yields a graph isomorphic to $\SLi{v+1}{w+1}{t}$ which is supposed not to occur by Hypothesis \ref{hypknmL}\eqref{hypkniii}.
The edges from $c_1$ to $a_i$ ($1 \leq i \leq v+1$) were dealt with above.
The graph induced by loss of the edge $\epsilon(c_1,c_i)$ for any $2\leq i \leq t+1$ does not occur by P\'alfy's condition given $c_1$, $c_i$, and $b_1$.
Hence, $c_1$ is admissible, and, by symmetry, $c_i$ is admissible for $2\leq i\leq t+1$. 

Since every vertex is admissible, we have that the graph $\SLi{v+1}{w+1}{(t+1)}$ does not occur by Theorem \ref{admissible}. Thus by induction, the graph $\SLi{k}{n}{m}$ does not occur as the prime character degree graph of any solvable group for all $k\geq3$ and $2\leq n\leq k$ and $m\geq2$.
\end{proof}

We can now state our result concerning the family $\{\SLi{k}{n}{m}\}$ in full generality.

\begin{thmA}
The graph $\SLi{k}{n}{m}$ occurs as the prime character degree graph of a solvable group when $(k,n,m)\in\{(1,1,i)~|~1\leq i\leq35\}$, and possibly when $(k,n,m)\in\{(1,1,i)~|~i>35\}$. Otherwise, $\SLi{k}{n}{m}$ does not occur as the prime character degree graph of any solvable group.
\end{thmA}
\begin{proof}
This unites the results from Theorem \ref{Lfam}, which handles $m=1$, along with the results from Propositions \ref{prop11mL}, \ref{prop21mL}, and \ref{propk1mL}, as well as Corollaries \ref{cor22mL} and \ref{corknmL}, which handle $m\geq2$.
\end{proof}

\section{The Right Family}\label{secright}

We next turn our attention to the family of graphs $\{\SRi{k}{n}{m}\}$. This collection of graphs has had extensive work done. In \cite{DLM}, the proof of the main result ($m=1$) required first classifying the graphs with $m\geq2$. We recall both of these below.

\begin{theorem}\label{Rfam}\emph{(\cite{DLM})}
The graph $\SR{k}{n}$ occurs as the prime character degree graph of a solvable group when $(k,n)=(1,1)$, and possibly when $(k,n)\in\{(2,1), (2,2)\}$. Otherwise, $\SR{k}{n}$ does not occur as the prime character degree graph of any solvable group.
\end{theorem}

\begin{proposition}\label{propknmR}\emph{(\cite{DLM})}
The graph $\SRi{k}{n}{m}$ does not occur as the prime character degree graph of any solvable group for all $k\geq3$ and $1\leq n\leq k$ and $m\geq2$.
\end{proposition}

We then notice that the graph $\SRi{k}{n}{m}$ is fully classified except for $k=1$ and $k=2$. For the former, one can see Figure \ref{fig11mR} for examples of graphs, which relate to Proposition \ref{prop11mR}.

\begin{figure}[htb]
    \centering
$
\begin{tikzpicture}[scale=2]
\node (xa1) at (0,0.5) {$a_1$};
\node (xb1) at (0.5,1) {$b_1$};
\node (xb2) at (0.5,0) {$b_2$};
\node (xc1) at (1,.8) {$c_1$};
\node (xc2) at (1,.2) {$c_2$};
\path[font=\small,>=angle 90]
(xa1) edge node [right] {$ $} (xb1)
(xa1) edge node [right] {$ $} (xb2)
(xb1) edge node [right] {$ $} (xb2)
(xb1) edge node [right] {$ $} (xc1)
(xb2) edge node [right] {$ $} (xc2)
(xc1) edge node [right] {$ $} (xc2)
(xb1) edge node [right] {$ $} (xc2)
(xb2) edge node [right] {$ $} (xc1);
\node (ya1) at (1.75,0.5) {$a_1$};
\node (yb1) at (2.25,0.8) {$b_1$};
\node (yb2) at (2.25,0.2) {$b_2$};
\node (yc1) at (2.75,1) {$c_1$};
\node (yc3) at (3.25,0.5) {$c_3$};
\node (yc2) at (2.75,0) {$c_2$};
\path[font=\small,>=angle 90]
(ya1) edge node [right] {$ $} (yb1)
(ya1) edge node [right] {$ $} (yb2)
(yb1) edge node [right] {$ $} (yc1)
(yb1) edge node [right] {$ $} (yb2)
(yb2) edge node [right] {$ $} (yc2)
(ya1) edge node [right] {$ $} (yb1)
(yc1) edge node [right] {$ $} (yc3)
(yc2) edge node [right] {$ $} (yc3)
(yb1) edge node [right] {$ $} (yc3)
(yb2) edge node [right] {$ $} (yc3)
(yc1) edge node [right] {$ $} (yc2)
(yc1) edge node [right] {$ $} (yb2)
(yc2) edge node [right] {$ $} (yb1);
\node (za1) at (4,0.5) {$a_1$};
\node (zb1) at (4.5,0.8) {$b_1$};
\node (zb2) at (4.5,0.2) {$b_2$};
\node (zc1) at (5,1) {$c_1$};
\node (zc2) at (5,0) {$c_2$};
\node (zc3) at (5.5,0.8) {$c_3$};
\node (zc4) at (5.5,0.2) {$c_4$};
\path[font=\small,>=angle 90]
(za1) edge node [right] {$ $} (zb1)
(za1) edge node [right] {$ $} (zb2)
(zb1) edge node [right] {$ $} (zc1)
(zb1) edge node [right] {$ $} (zb2)
(zb2) edge node [right] {$ $} (zc2)
(za1) edge node [right] {$ $} (zb1)
(zc1) edge node [right] {$ $} (zc3)
(zc2) edge node [right] {$ $} (zc3)
(zb1) edge node [right] {$ $} (zc3)
(zb2) edge node [right] {$ $} (zc3)
(zc1) edge node [right] {$ $} (zc2)
(zc1) edge node [right] {$ $} (zb2)
(zc2) edge node [right] {$ $} (zb1)
(zc4) edge node [right] {$ $} (zb2)
(zc4) edge node [right] {$ $} (zb1)
(zc4) edge node [right] {$ $} (zc1)
(zc4) edge node [right] {$ $} (zc2)
(zc4) edge node [right] {$ $} (zc3);
\end{tikzpicture}
$
    \caption{The graphs $\SRi{1}{1}{2}$, $\SRi{1}{1}{3}$, and $\SRi{1}{1}{4}$}
    \label{fig11mR}
\end{figure}
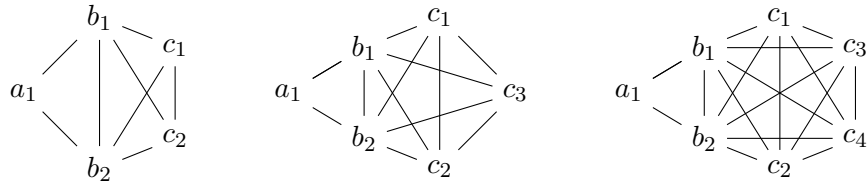

\begin{proposition}\label{prop11mR}
The graph $\SRi{1}{1}{m}$ occurs as the prime character degree graph of some solvable group for all $m\geq2$.
\end{proposition}
\begin{proof}
Let $m\in\mathbb{N}$ such that $m\geq2$. First, we recall that the singleton $K_1$ occurs as the prime character degree graph of some solvable group $G_1$, and so $\Delta(G_1)=K_1$. Next, notice that the graph $\Gamma_{m+1,1}$ occurs as the prime character degree graph of some solvable group $G_m$ by way of Theorem \ref{KT}. Thus, $\Delta(G_m)=\Gamma_{m+1,1}$. Therefore, considering the direct product $G_1\times G_m$, one observes that $\Delta(G_1\times G_m)=\SRi{1}{1}{m}$, whence the result.
\end{proof}

We note that the case of $k=2$ remains an issue. We are, as of yet, unable to determine if the graphs $\SRi{2}{n}{m}$ do or do not occur as the prime character degree graph of a solvable group. These graphs have been studied extensively, first appearing in \cite{BLL} as one of the nine unclassified graphs with six vertices, which in our notation is the graph $\SRi{2}{1}{1}=\SR{2}{1}$. These graphs with $k=2$ also show up in \cite{LMS} (with their label of $B_7=\SRi{2}{1}{2}$ and $B_{11}=\SRi{2}{2}{1}=\SR{2}{2}$) and in \cite{LS2} (with their label of $B_8=\SRi{2}{1}{3}$ and $B_{16}=\SRi{2}{2}{2}$), all of which went unclassified. Combining all relevant results, we are then left with the following theorem:

\begin{thmB}
The graph $\SRi{k}{n}{m}$ occurs as the prime character degree graph of a solvable group when $(k,n,m)\in\{(1,1,i)~|~i\in\mathbb{N}\}$, and possibly when $(k,n,m)\in\{(2,1,i), (2,2,i)~|~i\in\mathbb{N}\}$. Otherwise, $\SRi{k}{n}{m}$ does not occur as the prime character degree graph of any solvable group.
\end{thmB}
\begin{proof}
One combines the results from Theorem \ref{Rfam} ($m=1$), as well  as Proposition \ref{propknmR} ($k\geq3$ and $m\geq2$) and Proposition \ref{prop11mR} ($k=1$ and $m\geq2$).
\end{proof}

\section{The Zero Family}\label{seczero}

It is also worthwhile to look at the family of graphs when we stipulate the condition of $m=0$. In other words, we take zero ``special" vertices. Notationally, it is clear that since there are no vertices $c_i$ in this scenario, then the $L$ or $R$ superscript loses meaning. Therefore, we adopt the notational convention of
$$
\SZ{k}{n}:=\SLi{k}{n}{0}=\SRi{k}{n}{0}.
$$
Hence, the goal of this section is to classify this family $\{\SZ{k}{n}\}$. In order to do so, we consider three separate cases for $k$: either $1\leq k\leq2$, or $k=3$, or $k\geq4$.

\subsection{The cases of $k=1$ and $k=2$}

For these possibilities of $k=1$ and $k=2$, notice that due to the restriction of $1\leq n\leq k$, we have only three possible graphs that could arise: $\SZ{1}{1}$, $\SZ{2}{1}$, and $\SZ{2}{2}$. One can see Figure \ref{fig1and20} for these graphs.

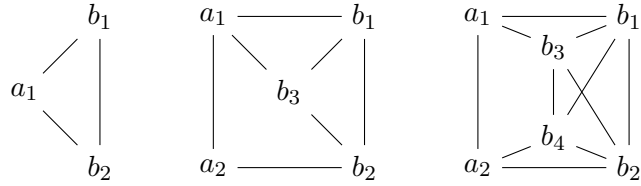
\begin{figure}[htb]
    \centering
$
\begin{tikzpicture}[scale=2]
\node (xa1) at (0,0.5) {$a_1$};
\node (xb1) at (0.5,1) {$b_1$};
\node (xb2) at (0.5,0) {$b_2$};
\path[font=\small,>=angle 90]
(xb1) edge node [right] {$ $} (xb2)
(xa1) edge node [right] {$ $} (xb1)
(xa1) edge node [right] {$ $} (xb2);
\node (ya1) at (1.25,1) {$a_1$};
\node (ya2) at (1.25,0) {$a_2$};
\node (yb1) at (2.25,1) {$b_1$};
\node (yb2) at (2.25,0) {$b_2$};
\node (yb3) at (1.75,.5) {$b_3$};
\path[font=\small,>=angle 90]
(ya1) edge node [right] {$ $} (ya2)
(yb1) edge node [right] {$ $} (yb2)
(yb1) edge node [right] {$ $} (yb3)
(yb2) edge node [right] {$ $} (yb3)
(ya1) edge node [right] {$ $} (yb1)
(ya2) edge node [right] {$ $} (yb2)
(ya1) edge node [right] {$ $} (yb3);
\node (za1) at (3,1) {$a_1$};
\node (za2) at (3,0) {$a_2$};
\node (zb1) at (4,1) {$b_1$};
\node (zb2) at (4,0) {$b_2$};
\node (zb3) at (3.5,.8) {$b_3$};
\node (zb4) at (3.5,.2) {$b_4$};
\path[font=\small,>=angle 90]
(za1) edge node [right] {$ $} (za2)
(zb1) edge node [right] {$ $} (zb2)
(zb1) edge node [right] {$ $} (zb3)
(zb1) edge node [right] {$ $} (zb4)
(zb2) edge node [right] {$ $} (zb3)
(zb2) edge node [right] {$ $} (zb4)
(zb3) edge node [right] {$ $} (zb4)
(za1) edge node [right] {$ $} (zb1)
(za2) edge node [right] {$ $} (zb2)
(za1) edge node [right] {$ $} (zb3)
(za2) edge node [right] {$ $} (zb4);
\end{tikzpicture}
$
    \caption{The graphs $\SZ{1}{1}$, $\SZ{2}{1}$, and $\SZ{2}{2}$}
    \label{fig1and20}
\end{figure}

All of the aforementioned graphs have been studied and classified before, which we recall below.

\begin{proposition}\label{prop1and20}\emph{(\cite{Z},\cite{L3},\cite{BLL})}
The graphs $\SZ{1}{1}$, $\SZ{2}{1}$, and $\SZ{2}{2}$ occur as the prime character degree graph of some solvable group.
\end{proposition}
\begin{proof}
Observe that the graph $\SZ{1}{1}=K_3$ and therefore has already been shown to occur as the prime character degree graph of some solvable group. In fact, one can see \cite{Z} for this classification, among others. Moreover, $\SZ{2}{1}$ is a graph with five vertices and so was classified in \cite{L3}. In that paper, Lewis classifies it as an occurring graph. Finally, the graph $\SZ{2}{2}$ was classified in \cite{BLL}, being a graph with six vertices. This graph was also shown to occur as the prime character degree graph of a solvable group, and the authors demonstrate this by employing the method of direct products.
\end{proof}

\subsection{The case of $k=3$}

We treat the case of $k=3$ separately for two main reasons. First, one of the graphs has already been classified, thus leaving only two graphs that need to be investigated (see Figure \ref{fig3n0} for the three total graphs). Second, it turns out that the argument needed to classify these graphs varies drastically from how we will work with the graphs within the case of $k\geq4$.

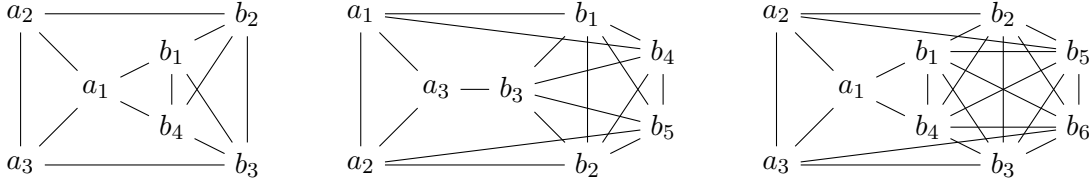
\begin{figure}[htb]
    \centering
$
\begin{tikzpicture}[scale=2]
\node (xa1) at (0,0.5) {$a_1$};
\node (xa2) at (-0.5,1) {$a_2$};
\node (xa3) at (-0.5,0) {$a_3$};
\node (xb1) at (0.5,0.75) {$b_1$};
\node (xb2) at (1,1) {$b_2$};
\node (xb3) at (1,0) {$b_3$};
\node (xb4) at (0.5,0.25) {$b_4$};
\path[font=\small,>=angle 90]
(xa1) edge node [right] {$ $} (xa2)
(xa1) edge node [right] {$ $} (xa3)
(xa2) edge node [right] {$ $} (xa3)
(xa1) edge node [right] {$ $} (xb1)
(xa1) edge node [right] {$ $} (xb4)
(xa2) edge node [right] {$ $} (xb2)
(xa3) edge node [right] {$ $} (xb3)
(xb1) edge node [right] {$ $} (xb2)
(xb1) edge node [right] {$ $} (xb3)
(xb1) edge node [right] {$ $} (xb4)
(xb2) edge node [right] {$ $} (xb3)
(xb2) edge node [right] {$ $} (xb4)
(xb3) edge node [right] {$ $} (xb4);
\node (ya3) at (2.25,0.5) {$a_3$};
\node (ya1) at (1.75,1) {$a_1$};
\node (ya2) at (1.75,0) {$a_2$};
\node (yb4) at (3.75,0.75) {$b_4$};
\node (yb1) at (3.25,1) {$b_1$};
\node (yb2) at (3.25,0) {$b_2$};
\node (yb5) at (3.75,0.25) {$b_5$};
\node (yb3) at (2.75,0.5) {$b_3$};
\path[font=\small,>=angle 90]
(ya1) edge node [right] {$ $} (ya2)
(ya1) edge node [right] {$ $} (ya3)
(ya2) edge node [right] {$ $} (ya3)
(ya1) edge node [right] {$ $} (yb1)
(ya1) edge node [right] {$ $} (yb4)
(ya2) edge node [right] {$ $} (yb2)
(ya2) edge node [right] {$ $} (yb5)
(ya3) edge node [right] {$ $} (yb3)
(yb1) edge node [right] {$ $} (yb2)
(yb1) edge node [right] {$ $} (yb3)
(yb1) edge node [right] {$ $} (yb4)
(yb1) edge node [right] {$ $} (yb5)
(yb2) edge node [right] {$ $} (yb3)
(yb2) edge node [right] {$ $} (yb4)
(yb2) edge node [right] {$ $} (yb5)
(yb3) edge node [right] {$ $} (yb4)
(yb3) edge node [right] {$ $} (yb5)
(yb4) edge node [right] {$ $} (yb5);
\node (za1) at (5.00,0.5) {$a_1$};
\node (za2) at (4.50,1) {$a_2$};
\node (za3) at (4.50,0) {$a_3$};
\node (zb1) at (5.50,0.75) {$b_1$};
\node (zb2) at (6.00,1) {$b_2$};
\node (zb3) at (6.00,0) {$b_3$};
\node (zb4) at (5.50,0.25) {$b_4$};
\node (zb5) at (6.50,0.75) {$b_5$};
\node (zb6) at (6.50,0.25) {$b_6$};
\path[font=\small,>=angle 90]
(za1) edge node [right] {$ $} (za2)
(za1) edge node [right] {$ $} (za3)
(za2) edge node [right] {$ $} (za3)
(za1) edge node [right] {$ $} (zb1)
(za1) edge node [right] {$ $} (zb4)
(za2) edge node [right] {$ $} (zb2)
(za2) edge node [right] {$ $} (zb5)
(za3) edge node [right] {$ $} (zb3)
(za3) edge node [right] {$ $} (zb6)
(zb1) edge node [right] {$ $} (zb2)
(zb1) edge node [right] {$ $} (zb3)
(zb1) edge node [right] {$ $} (zb4)
(zb1) edge node [right] {$ $} (zb5)
(zb1) edge node [right] {$ $} (zb6)
(zb2) edge node [right] {$ $} (zb3)
(zb2) edge node [right] {$ $} (zb4)
(zb2) edge node [right] {$ $} (zb5)
(zb2) edge node [right] {$ $} (zb6)
(zb3) edge node [right] {$ $} (zb4)
(zb3) edge node [right] {$ $} (zb5)
(zb3) edge node [right] {$ $} (zb6)
(zb4) edge node [right] {$ $} (zb5)
(zb4) edge node [right] {$ $} (zb6)
(zb5) edge node [right] {$ $} (zb6) ;
\end{tikzpicture}
$
    \caption{The graphs $\SZ{3}{1}$, $\SZ{3}{2}$, and $\SZ{3}{3}$}
    \label{fig3n0}
\end{figure}

\begin{proposition}\label{prop310}\emph{(\cite{LMS})}
The graph $\SZ{3}{1}$ does not occur as the prime character degree graph of any solvable group.
\end{proposition}
\begin{proof}
As demonstrated in \cite{LMS}, the graph $C_{20}=\SZ{3}{1}$ does not occur as the prime character degree graph of any solvable group.
\end{proof}

We now turn our attention to the graph $\SZ{3}{2}$, which is a graph with eight vertices and therefore makes an appearance in \cite{LS2}. One can see Figure \ref{fig320} for this graph. In order to classify $\SZ{3}{2}$, we will follow an argument similar to how $\SZ{3}{1}$ was handled in \cite{LMS}, which actually mimics the series of lemmas used to prove Proposition \ref{prop21mL}.

\begin{figure}[htb]
    \centering
$
\begin{tikzpicture}[scale=2]
\node (ya3) at (2.25,0.5) {$a_3$};
\node (ya1) at (1.75,1) {$a_1$};
\node (ya2) at (1.75,0) {$a_2$};
\node (yb4) at (3.75,0.75) {$b_4$};
\node (yb1) at (3.25,1) {$b_1$};
\node (yb2) at (3.25,0) {$b_2$};
\node (yb5) at (3.75,0.25) {$b_5$};
\node (yb3) at (2.75,0.5) {$b_3$};
\path[font=\small,>=angle 90]
(ya1) edge node [right] {$ $} (ya2)
(ya1) edge node [right] {$ $} (ya3)
(ya2) edge node [right] {$ $} (ya3)
(ya1) edge node [right] {$ $} (yb1)
(ya1) edge node [right] {$ $} (yb4)
(ya2) edge node [right] {$ $} (yb2)
(ya2) edge node [right] {$ $} (yb5)
(ya3) edge node [right] {$ $} (yb3)
(yb1) edge node [right] {$ $} (yb2)
(yb1) edge node [right] {$ $} (yb3)
(yb1) edge node [right] {$ $} (yb4)
(yb1) edge node [right] {$ $} (yb5)
(yb2) edge node [right] {$ $} (yb3)
(yb2) edge node [right] {$ $} (yb4)
(yb2) edge node [right] {$ $} (yb5)
(yb3) edge node [right] {$ $} (yb4)
(yb3) edge node [right] {$ $} (yb5)
(yb4) edge node [right] {$ $} (yb5);
\end{tikzpicture}
$
    \caption{The graph $\SZ{3}{2}$}
    \label{fig320}
\end{figure}
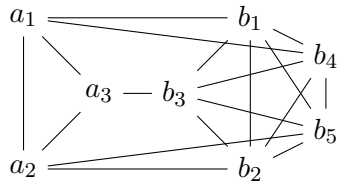

\begin{lemma}\label{lem32ad}
Suppose that $\SZ{3}{2}=\Delta(G)$ for some finite solvable group $G$, where $|G|$ is minimal. Then $G$ does not have any normal nonabelian Sylow $b_i$-subgroup for all $1\leq i\leq5$. In particular, $b_i$ is a strongly admissible vertex.
\end{lemma}
\begin{proof}
First, let us consider removing the vertex $b_1$ and all incident edges. Doing so produces graph $\SZ{3}{1}$ referenced in Proposition \ref{prop310}, which does not occur. Now we shall examine the removal of incident edges to $b_1$. Removing the edge $\epsilon(b_1,b_i)$ for any $2\leq i\leq5$ gives us a graph that violates P\'alfy's condition with vertices $b_1$, $b_i$, and $a_j$ for some $2\leq j\leq3$ with $j\neq i$. Therefore, none of these edges can be lost. The only other edge to consider, therefore, is the one between $b_1$ and $a_1$, whose removal yields $\SR{3}{1}$. This graph is shown not to occur by Theorem \ref{Rfam}. Thus, $b_1$ is admissible.

Next we verify that $b_1$ satisfies the third condition of Definition \ref{adstrong}. We start with the removal of $b_1$ and all incident edges, and then the possible removal of the edge $\epsilon(b_i,b_j)$ for any $2\leq i<j\leq5$. Notice that by losing the aforementioned edge, we thereby violate P\'alfy's condition with the threesome $b_i$, $b_j$, and $a_l$ for some $1\leq l\leq3$ such that $l\neq i$ and $l\neq j$. Therefore none of these edges can be lost. The only other edge to consider removing is $\epsilon(a_1,b_4)$, but the consequential graph contradicts Sass's main result from \cite{S}. One can see this by denoting $p=a_1$ and $q=b_4$, forcing $|\rho_1\cup\rho_2|=3$ but $4=|\rho_3\cup\rho_4|\geq2^3=8$, a contradiction to Theorem \ref{diameter}\eqref{diam3}. Thus, $b_1$ is strongly admissible. Due to symmetry, we also get that $b_2$, $b_4$, and $b_5$ are strongly admissible.

Next, we consider removing the vertex $b_3$ and all incident edges. Doing so yields the graph $\SL{2}{2}$ from Theorem \ref{Lfam}, which does not occur. Next, we consider removing incident edges only. It was shown above that none of the edges $\epsilon(b_3,b_i)$ can be removed for any $1\leq i\leq5$ with $i\neq3$ due to P\'alfy's condition. Finally, removing the edge $\epsilon(b_3,a_3)$ reduces $\SZ{3}{2}$ to a graph of diameter three which does not occur due to Theorem \ref{diameter}. Thus, $b_3$ is admissible.

Next, we consider the graph determined by the removal of $b_3$ and all its incident edges, along with possibly the edges between $b_i$ and $b_j$ for any $1\leq i<j\leq5$ with $i\neq3$ and $j\neq3$. This will clearly violate P\'alfy's condition as above. Therefore, $b_3$ fulfills the third condition of Definition \ref{adstrong} and thus is strongly admissible.

Since $b_i$ is strongly admissible for all $1\leq i\leq5$, then by Lemma \ref{strong}, there is no normal nonabelian Sylow $b_i$-subgroup, as desired.
\end{proof}

We next address the three remaining vertices of the graph $\SZ{3}{2}$, which are the primes $a_1$, $a_2$, and $a_3$.

\begin{lemma}\label{lem32pi}
Suppose that $\SZ{3}{2}=\Delta(G)$ for some finite solvable group $G$, where $|G|$ is minimal. Then $G$ does not have any normal nonabelian Sylow $a_i$-subgroup for all $1\leq i\leq3$.
\end{lemma}
\begin{proof}
Following the notation of Lemma \ref{pi}, let us first consider $q=a_1$, which leads to $\pi=\{a_2,a_3,b_1,b_4\}$ and $\rho=\{b_2,b_3,b_5\}$. We can then denote $\pi_1=\{b_1,b_4\}$ and $\pi_2=\{a_2,a_3\}$. This allows us to consider $v=a_2\in\pi_2$, $s=b_2\in\rho$, and $w=b_3\in\rho$, where one can see that $a_2$ is adjacent to $b_2$ but not adjacent to $b_3$. Finally, note that $s=b_2$ is indeed admissible, as shown in Lemma \ref{lem32ad}. With all the conditions fulfilled, the result of Lemma \ref{pi} follows, and so there is no normal nonabelian Sylow $a_1$-subgroup. One can follow an identical argument with $a_2$ due to symmetry.

Next, we examine the situation when $q=a_3$, which leads to $\pi=\{a_1,a_2,b_3\}$ and $\rho=\{b_1,b_2,b_4,b_5\}$. We then separate $\pi$ into $\pi_1=\{b_3\}$ and $\pi_2=\{a_1,a_2\}$. This permits the consideration of $v=a_1\in\pi_2$ and $s=b_1\in\rho$, where we know that $s$ is admissible by Lemma \ref{lem32ad}. Lastly, we denote $w=b_2\in\rho$, which is not adjacent to $v=a_1$. Hence, again by Lemma \ref{pi}, the result follows.
\end{proof}

\begin{lemma}\label{lem32sub}
No proper subgraph of $\SZ{3}{2}$ with the same vertex set occurs as the prime character degree graph of any solvable group.
\end{lemma}
\begin{proof}
While showing admissibility, we have demonstrated that none of the edges in the complete subgraph on $\{b_1,b_2,b_3,b_4,b_5\}$ can be removed. Likewise, it is easy to see that we cannot lose any of the edges in the complete subgraph on $\{a_1,a_2,a_3\}$. This is due to P\'alfy's condition. Therefore, there are only five edges remaining to be considered, which are:
\begin{equation}\label{edges}
\epsilon(a_1,b_1),~\epsilon(a_1,b_4),~\epsilon(a_2,b_2),~\epsilon(a_2,b_5),\text{~and~}\epsilon(a_3,b_3).
\end{equation}
We shall analyze the subgraphs determined by removing any of the aforementioned edges, on their own or in a combination, and show they cannot occur. Before we begin, recall that $b_1$ and $b_4$ are symmetric, meaning edges $\epsilon(a_1,b_1)$ and $\epsilon(a_1,b_4)$ are symmetric as well. Likewise, $b_2$ and $b_5$ are symmetric, which means the edges $\epsilon(a_2,b_2)$ and $\epsilon(a_2,b_5)$ are also symmetric. Further, the vertices $a_1$ and $a_2$ are symmetric too.

Now, it was shown in Lemma \ref{lem32ad} that none of the above edges could be lost on its own. We proceed to consider losing a pair of the edges from \eqref{edges}. Removing any two of the aforementioned edges yields either the graph $C_{13}$ from \cite{LS2} or the graph $\Gamma_{5,3}$ (which is the graph $C_{14}$ from \cite{LS2}), both of which have been shown not to occur. Likewise, the loss of three edges from \eqref{edges} results in either the graph $C_{4}$ or the graph $C_{5}$ from \cite{LS2}, both of which the authors classified as non-occurring. Next, the removal of four of the edges mentioned above results in the graph $C_{1}$ from \cite{LS2}, which was shown not to occur as the prime character degree graph of any solvable group. Finally, losing all five edges of \eqref{edges} yields the disconnected graph of $K_3$ and $K_5$, which the authors of \cite{LS2} showed not to occur by way of P\'alfy's inequality from \cite{P2}.

Hence, no proper subgraph of $\SZ{3}{2}$ with the same vertex set occurs as the prime character degree graph of any solvable group.
\end{proof}

Under the assumption that there is some finite solvable group $G$ such that $\Delta(G)=\SZ{3}{2}$ such that $|G|$ is minimal, we now have the conclusion that there are no normal nonabelian Sylow $p$-subgroups for all $p\in\rho(G)$. This is due to Lemmas \ref{lem32ad} and \ref{lem32pi}. Equipped with the Fitting subgroup $F$ of $G$, we then have that $\rho(G)=\pi(|G:F|)$, which implies $\rho(G)=\rho(G/\Phi(G))$. However, due to Lemma \ref{lem32sub}, we are then forced to conclude that $\Phi(G)=1$, leading to the existence of a subgroup $H$ of $G$ such that $G=HF$ such that $H\cap F=1$ (via Lemma III 4.4 from \cite{H}). As before, our next task is to verify that $F$ is in fact minimal normal in $G$.

\begin{lemma}\label{lem32min}
Suppose that $\SZ{3}{2}=\Delta(G)$ for some finite solvable group $G$, where $|G|$ is minimal. Then the Fitting subgroup $F$ of $G$ is minimal normal in $G$.
\end{lemma}
\begin{proof}
We follow an argument similar to what was done in Lemma \ref{lem21mLfit}. We adopt the notation from the previous paragraph and further set $E$ to be the Fitting subgroup of $H$.

For the sake of contradiction, we suppose that $F$ is not minimal normal in $G$. Therefore, there must exist some normal subgroup $N$ of $G$ such that $1<N<F$.

By Lemma III 4.5 from \cite{H}, we see that there exists a nontrivial normal subgroup $M$ of $G$ such that $F=N\times M$, and consequently
\begin{equation}\label{eqsub}
\rho(G/N)\subset\rho(G)\text{~and~}\rho(G/M)\subset\rho(G).  
\end{equation}
We make note that these are indeed proper subsets. Therefore, we may now consider an arbitrary $q\in\rho(G)\setminus\rho(G/N)$. It is known that $G/N$ has a normal nonabelian Sylow $q$-subgroup whose class is a formation, and so this implies that $q\in\rho(G/M)$. This forces the equality of
\begin{equation}\label{equn}
\rho(G)=\rho(G/N)\cup\rho(G/M).
\end{equation}
In particular, let $q\in\rho(G)\setminus\rho(G/N)$. Since $q\in\rho(G)=\pi(|G:F|)$ and $G/F\cong H$, then we have that $q\mid|H|$. Moreover, by the usual formation argument, we then get that $q\in\rho(G/M)$. Hence, the primes in $\rho(G)\setminus[\rho(G/N)\cap\rho(G/M)]$ arise from the relevant Hall subgroup of $H$, and the character degree argument involving $E$ forces this set to lie in a complete subgraph of $\Delta(G)$. This leaves us with five cases to consider. So $\rho(G)\setminus[\rho(G/N)\cap\rho(G/M)]$ must be contained in $A=\{a_1,a_2,a_3\}$, $B=\{a_1,b_1,b_4\}$, $\widehat{B}=\{a_2,b_2,b_5\}$, $C=\{a_3,b_3\}$, or $D=\{b_1,b_2,b_3,b_4,b_5\}$.

First suppose that $\rho(G)\setminus[\rho(G/N)\cap\rho(G/M)]\subseteq A$. We follow the argument that Lewis presents in \cite{L3}, and see that $E$, the Fitting subgroup of $H$, has a Hall $A$-subgroup of $H$. Since $\cd(G)$ has a character degree that is divisible by all the primes dividing $|E|$, and since $|E|$ is divisible by no other primes, then we know that $E$ is in fact \emph{the} Hall $A$-subgroup of $H$. Next, let $\chi\in\Irr(G)$ such that $a_1a_2a_3$ divides $\chi(1)$, and let $\theta$ be an irreducible constituent of $\chi_{FE}$. One can verify that $\chi(1)/\theta(1)$ divides $|G:FE|$, which forces $\chi_{FE}=\theta$. Since $a_1$, $a_2$, and $a_3$ all divide $\theta(1)$, and the only possible primes that could divide an element in the set of character degrees for $G/FE$ are $b_i$ for $1\leq i\leq5$, then by Gallagher's Theorem we get that $\cd(G/FE)=\{1\}$, and so $G/FE$ is abelian. Since $G/FE$ is abelian, then $O^{b_i}(G)<G$ for some $b_i$, a contradiction because $O^{b_i}(G)=G$ by Lemma \ref{bigoh} since each $b_i$ is admissible by Lemma \ref{lem32ad}. Hence, this case cannot occur.

Next, suppose that $\rho(G)\setminus[\rho(G/N)\cap\rho(G/M)]\subseteq B$. It is easy to see that this implies that $\{a_2,a_3,b_2,b_3,b_5\}$ must be contained in both $\rho(G/N)$ and $\rho(G/M)$. Since $a_1\in\rho(G)$ and due to \eqref{equn}, then without loss of generality we may suppose that $a_1\in\rho(G/N)$. Due to \eqref{eqsub}, we know $\rho(G/N)$ is proper in $\rho(G)$, and so there are only three options for this vertex set: (i) $\{a_1,a_2,a_3,b_2,b_3,b_5\}$, (ii) $\{a_1,a_2,a_3,b_1,b_2,b_3,b_5\}$, or (iii) $\{a_1,a_2,a_3,b_2,b_3,b_4,b_5\}$. Supposing case (i), we have that $\rho(G/N)=\{a_1,a_2,a_3,b_2,b_3,b_5\}$. One can see that no subgraph, connected or disconnected, occurs with this vertex set, and therefore this case cannot happen. For case (ii), we are left with $\rho(G/N)=\rho(G)\setminus\{b_4\}$, and here one again can verify that no subgraph occurs with this vertex set. Therefore, this case cannot happen either. Finally, since $b_1$ and $b_4$ are symmetric vertices, we see that case (iii) is identical to case (ii). Hence, the case of $\rho(G)\setminus[\rho(G/N)\cap\rho(G/M)]\subseteq B$ must not happen.

Due to symmetry, notice that supposing $\rho(G)\setminus[\rho(G/N)\cap\rho(G/M)]\subseteq\widehat{B}$ follows identically to the case of $\rho(G)\setminus[\rho(G/N)\cap\rho(G/M)]\subseteq B$ handled above.

Next, suppose that $\rho(G)\setminus[\rho(G/N)\cap\rho(G/M)]\subseteq C$. This forces $\{a_1,a_2,b_1,b_2,b_4,b_5\}\subseteq\rho(G/N)\cap\rho(G/M)$. Again due to \eqref{eqsub} and \eqref{equn}, we are pigeon-holed into one case, which we can without loss of generality label as
$$
\rho(G/N)=\{a_1,a_2,a_3,b_1,b_2,b_4,b_5\}\text{~and~}\rho(G/M)=\{a_1,a_2,b_1,b_2,b_3,b_4,b_5\}.
$$
However, notice that any subgraph of $\Delta(G)$ with the vertex set of $\rho(G/N)$ described above, either connected or disconnected, does not occur as the prime character degree graph of any solvable group. Thus, this case cannot happen.

Finally, we are then left with the case of $\rho(G)\setminus[\rho(G/N)\cap\rho(G/M)]\subseteq D$, which implies $\{a_1,a_2,a_3\}\subseteq\rho(G/N)\cap\rho(G/M)$. Since $b_3\in\rho(G)$ and due to \eqref{equn}, without loss of generality we may suppose that $b_3\in\rho(G/N)$. Furthermore, due to \eqref{eqsub}, we know that $\rho(G/N)\neq\rho(G)$. Since $b_1$ and $b_4$ are symmetric, as are $b_2$ and $b_5$, we then have the following distinct non-symmetric options for the vertex set $\rho(G/N)$: (i) $\{a_1,a_2,a_3,b_3\}$, (ii) $\{a_1,a_2,a_3,b_1,b_3\}$, (iii) $\{a_1,a_2,a_3,b_1,b_2,b_3\}$, (iv) $\{a_1,a_2,a_3,b_1,b_3,b_4\}$, or (v) $\{a_1,a_2,a_3,b_1,b_2,b_3,b_4\}$. For (i), supposing $\rho(G/N)=\{a_1,a_2,a_3,b_3\}$ forces $\rho(G/M)=\{a_1,a_2,a_3,b_1,b_2,b_4,b_5\}$ due to \eqref{eqsub} and \eqref{equn}. However, as mentioned in case $C$ above, no graph occurs with this vertex set. Thus, this case cannot happen. For option (ii), the only graph that could occur with this vertex set is the disconnected graph with complete components on the vertex sets $\{a_1,a_2,a_3\}$ and $\{b_1,b_3\}$. Then by Theorem 5.5 from \cite{L} we conclude that $G/N$ must have a central Sylow $b_2$-subgroup, a central Sylow $b_4$-subgroup, or a central Sylow $b_5$-subgroup, in which case
$$
O^{b_2}(G)<G\text{~or~}O^{b_4}(G)<G\text{~or~}O^{b_5}(G)<G.
$$
Regardless, this will lead to a contradiction, as each of those vertices has been shown to be admissible in Lemma \ref{lem32ad}, therefore forcing
$$
O^{b_2}(G)=G\text{~and~}O^{b_4}(G)=G\text{~and~}O^{b_5}(G)=G
$$
by way of Lemma \ref{bigoh}. Hence, this case also cannot occur. For (iii), we get $\rho(G/N)=\{a_1,a_2,a_3,b_1,b_2,b_3\}$, and it is easy to see that neither a connected nor a disconnected subgraph with that vertex set can occur, and so this case cannot happen. For (iv), though distinct from (iii) above, one sees that the conclusion is the same. Finally, supposing (v), we get that $\rho(G/N)=\rho(G)\setminus\{b_5\}$, and no subgraph with this vertex set occurs, either connected or disconnected. Hence, we must have that $\rho(G)\setminus[\rho(G/N)\cap\rho(G/M)]\subseteq D$ cannot happen.

We have exhausted all options for the set $\rho(G)\setminus[\rho(G/N)\cap\rho(G/M)]$, and so our original supposition must have been false. Therefore, the Fitting subgroup $F$ must be minimal normal in $G$.
\end{proof}

\begin{proposition}\label{prop320}
The graph $\SZ{3}{2}$ does not occur as the prime character degree graph of any solvable group.
\end{proposition}
\begin{proof}
For the sake of contradiction, suppose that $\SZ{3}{2}=\Delta(G)$ for some finite solvable group $G$, where $|G|$ is minimal. Due to Lemma \ref{lem32min}, the Fitting subgroup $F$ of $G$ is minimal normal in $G$, and so using the notation of Proposition \ref{final}, we take $a=a_1$, $b=a_2$, $c=b_1$, and $d=b_2$, noting that $b_1$ and $b_2$ are admissible by Lemma \ref{lem32ad}. Proposition \ref{final} grants the contradiction, and hence the graph $\SZ{3}{2}$ must not occur as the prime character degree graph of any solvable group.
\end{proof}

As mentioned above, the graph $\SZ{3}{2}$ has eight vertices. In particular, one can reconcile $\SZ{3}{2}$ as the graph $C_{95}$ from \cite{LS2}. Using Proposition \ref{prop320}, we can now officially classify the graph $\SZ{3}{2}$ ($C_{95}$) as non-occurring.

\begin{corollary}\label{cor330}
The graph $\SZ{3}{3}$ does not occur as the prime character degree graph of any solvable group.
\end{corollary}
\begin{proof}
This follows similarly to the series of lemmas used to prove Proposition \ref{prop320}. We omit the full argument here, but using the graph of $\SZ{3}{3}$ as seen in Figure \ref{fig3n0} as reference, one can verify that $b_i$ is strongly admissible for all $1\leq i\leq6$, and therefore there is no normal nonabelian Sylow $b_i$-subgroup due to Lemma \ref{strong}. Moreover, following Lemma \ref{pi}, one can show that there is no normal nonabelian Sylow $a_j$-subgroup for all $1\leq j\leq3$. To this end, one can then check that no proper subgraph of $\SZ{3}{3}$ with the same vertex set occurs as the prime character degree graph of any solvable group, all of which can allow us to conclude that the corresponding Frattini subgroup is trivial. One then proceeds to check that the Fitting subgroup is minimal normal, thereby setting up the final contradiction by way of Proposition \ref{final}, employing the vertices $a=a_1$, $b=a_2$, $c=b_1$, and $d=b_2$.
\end{proof}

\subsection{The case of $k\geq4$}%

For graphs with $k\geq4$, we will show that they do not occur by showing every vertex is admissible. We break this into two arguments, the first (Lemma \ref{lemk10}) being the base case of an induction argument, and the second (Proposition \ref{propkn0}) being the inductive step. One can see examples of these graphs in Figure \ref{figk10} and Figure \ref{figkn0}, respectively.

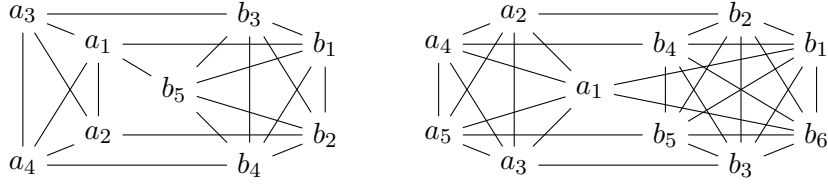
\begin{figure}[htb]
    \centering
$
\begin{tikzpicture}[scale=2]
\node (xa1) at (0,0.8) {$a_1$};
\node (xa2) at (0,0.2) {$a_2$};
\node (xa3) at (-0.5,1) {$a_3$};
\node (xa4) at (-0.5,0) {$a_4$};
\node (xb1) at (1.5,0.8) {$b_1$};
\node (xb2) at (1.5,0.2) {$b_2$};
\node (xb3) at (1,1) {$b_3$};
\node (xb4) at (1,0) {$b_4$};
\node (xb5) at (0.5,0.5) {$b_5$};
\path[font=\small,>=angle 90]
(xa1) edge node [right] {$ $} (xa2)
(xa1) edge node [right] {$ $} (xa3)
(xa1) edge node [right] {$ $} (xa4)
(xa2) edge node [right] {$ $} (xa3)
(xa2) edge node [right] {$ $} (xa4)
(xa3) edge node [right] {$ $} (xa4)
(xa1) edge node [right] {$ $} (xb1)
(xa2) edge node [right] {$ $} (xb2)
(xa3) edge node [right] {$ $} (xb3)
(xa4) edge node [right] {$ $} (xb4)
(xa1) edge node [right] {$ $} (xb5)
(xb1) edge node [right] {$ $} (xb2)
(xb1) edge node [right] {$ $} (xb3)
(xb1) edge node [right] {$ $} (xb4)
(xb1) edge node [right] {$ $} (xb5)
(xb2) edge node [right] {$ $} (xb3)
(xb2) edge node [right] {$ $} (xb4)
(xb2) edge node [right] {$ $} (xb5)
(xb3) edge node [right] {$ $} (xb4)
(xb3) edge node [right] {$ $} (xb5)
(xb4) edge node [right] {$ $} (xb5);
\node (ya1) at (3.25,0.5) {$a_1$};
\node (ya2) at (2.75,1) {$a_2$};
\node (ya3) at (2.75,0) {$a_3$};
\node (ya4) at (2.25,.8) {$a_4$};
\node (ya5) at (2.25,.2) {$a_5$};
\node (yb1) at (4.75,0.8) {$b_1$};
\node (yb2) at (4.25,1) {$b_2$};
\node (yb3) at (4.25,0) {$b_3$};
\node (yb4) at (3.75,.8) {$b_4$};
\node (yb5) at (3.75,.2) {$b_5$};
\node (yb6) at (4.75,0.2) {$b_6$};
\path[font=\small,>=angle 90]
(ya1) edge node [right] {$ $} (ya2)
(ya1) edge node [right] {$ $} (ya3)
(ya1) edge node [right] {$ $} (ya4)
(ya1) edge node [right] {$ $} (ya5)
(ya2) edge node [right] {$ $} (ya3)
(ya2) edge node [right] {$ $} (ya4)
(ya2) edge node [right] {$ $} (ya5)
(ya3) edge node [right] {$ $} (ya4)
(ya3) edge node [right] {$ $} (ya5)
(ya4) edge node [right] {$ $} (ya5)
(ya1) edge node [right] {$ $} (yb1)
(ya1) edge node [right] {$ $} (yb6)
(ya2) edge node [right] {$ $} (yb2)
(ya3) edge node [right] {$ $} (yb3)
(ya4) edge node [right] {$ $} (yb4)
(ya5) edge node [right] {$ $} (yb5)
(yb1) edge node [right] {$ $} (yb2)
(yb1) edge node [right] {$ $} (yb3)
(yb1) edge node [right] {$ $} (yb4)
(yb1) edge node [right] {$ $} (yb5)
(yb1) edge node [right] {$ $} (yb6)
(yb2) edge node [right] {$ $} (yb3)
(yb2) edge node [right] {$ $} (yb4)
(yb2) edge node [right] {$ $} (yb5)
(yb2) edge node [right] {$ $} (yb6)
(yb3) edge node [right] {$ $} (yb4)
(yb3) edge node [right] {$ $} (yb5)
(yb3) edge node [right] {$ $} (yb6)
(yb4) edge node [right] {$ $} (yb5)
(yb4) edge node [right] {$ $} (yb6)
(yb5) edge node [right] {$ $} (yb6);
\end{tikzpicture}
$
    \caption{The graphs $\SZ{4}{1}$ and $\SZ{5}{1}$}
    \label{figk10}
\end{figure}

\begin{lemma}\label{lemk10}
The graph $\SZ{k}{1}$ does not occur as the prime character degree graph of any solvable group for all $k \geq 4$.
\end{lemma}
\begin{proof}
Let $k\in\mathbb{N}$ such that $k\geq 4$. We aim to show that the graph $\SZ{k}{1}$ does not occur as the prime character degree graph of any solvable group. We will do this by showing that every vertex is admissible.

First, we consider the vertex $a_1$. Observe that the graph induced by the loss of $a_1$ and all incident edges is isomorphic to the graph $\Gamma_{k+1,k-1}$, which is known not to occur by Theorem \ref{KT}. Next, we consider removing edges incident to $a_1$. The loss of the edge $\epsilon(a_1,a_i)$ for any $2\leq i\leq k$ results in a graph which does not occur by P\'alfy's condition. One can see this by using the trio $a_1$, $a_i$, and $b_j$ for some $2\leq j\leq k$ with $i\neq j$. Hence, the edge between $a_1$ and $a_i$ cannot be lost for all $2\leq i\leq k$. Furthermore, losing one or both of the edges $\epsilon(a_1,b_1)$ and $\epsilon(a_1,b_{k+1})$ results in the graph $\Gamma_{k+1,k}$ (or a proper connected subgraph of it with the same vertex set), and cannot occur again by Theorem \ref{KT}. Hence, $a_1$ is admissible.

Next, we consider the prime $a_2$. Observe that removing $a_2$ and all incident edges results in the graph $\SZ{k+1}{1}[a_2]=\SR{k-1}{1}$ which does not occur by Theorem \ref{Rfam}. Therefore, the first part of Definition \ref{adstrong} is satisfied. As for the second part, notice that the edge $\epsilon(a_2,a_i)$ cannot be lost for all $1\leq i\leq k$ with $i\neq2$ due to an argument with P\'alfy's condition, as considered above. The only edge that is left to consider losing, therefore, is the one between $a_2$ and $b_2$. Losing this edge yields a graph of diameter three which will violate the main result from \cite{S} (see Theorem \ref{diameter}\eqref{diam3} for example). Hence, $a_2$ is admissible. By symmetry, we also get that $a_i$ is admissible for all $3\leq i\leq k$.

Considering losing the vertex $b_1$ and all incident edges gives the graph $\Gamma_{k,k}$, which is known not to occur by Theorem \ref{KT}. Next, we investigate the loss of edges incident to $b_1$. First, we notice that removing the edge $\epsilon(b_1,b_i)$ for any $2\leq i\leq k+1$ violates P\'alfy's condition with $b_1$, $b_i$, and $a_j$ for some $2\leq j\leq k$ with $i\neq j$. Hence, these edges cannot be lost. Removing the remaining edge between $b_1$ and $a_1$ results in the graph $\Gamma_{k+1,k}$, which does not occur as the prime character degree graph of any solvable group by Theorem \ref{KT}. Thus, $b_1$ is admissible, and by symmetry, so is $b_{k+1}$.

Finally, we consider $b_2$. Losing $b_2$ and all incident edges induces a graph isomorphic to the non-occurring graph $\SL{k-1}{1}$ (see Theorem \ref{Lfam}). This verifies the first part of Definition \ref{adstrong}. As for the second part, the removal of edges incident to $b_2$ (that is, the edges $\epsilon(b_2,b_i)$ for all $1\leq i\leq k+1$ with $i\neq2$, and the edge $\epsilon(b_2,a_2)$) were all considered above and were shown that they cannot be lost. Thus, $b_2$ is admissible. By a symmetrical argument, one can glean that $b_i$ is also admissible for all $3\leq i\leq k$.

Since every vertex is admissible, then, by way of Theorem \ref{admissible}, we can conclude that the graph $\SZ{k}{1}$ does not occur as the prime character degree graph of any solvable group, which was what we wanted.
\end{proof}

As mentioned above, we will be viewing Lemma \ref{lemk10} as a base case of an induction argument. We now turn our attention to letting the index $n$ range, which is $1\leq n\leq k$ (see Figure \ref{figkn0} for examples).

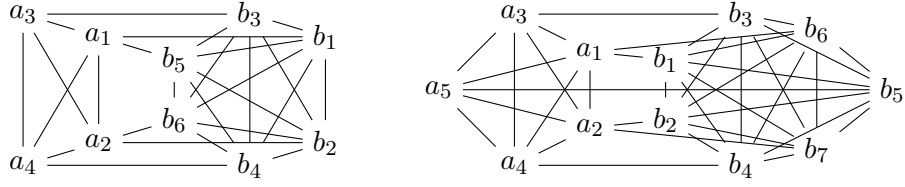
\begin{figure}[htb]
    \centering
$
\begin{tikzpicture}[scale=2]
\node (xa1) at (0,0.85) {$a_1$};
\node (xa2) at (0,0.15) {$a_2$};
\node (xa3) at (-0.5,1) {$a_3$};
\node (xa4) at (-0.5,0) {$a_4$};
\node (xb1) at (1.5,0.85) {$b_1$};
\node (xb2) at (1.5,0.15) {$b_2$};
\node (xb3) at (1,1) {$b_3$};
\node (xb4) at (1,0) {$b_4$};
\node (xb5) at (0.5,0.7) {$b_5$};
\node (xb6) at (0.5,0.3) {$b_6$};
\path[font=\small,>=angle 90]
(xa1) edge node [right] {$ $} (xa2)
(xa1) edge node [right] {$ $} (xa3)
(xa1) edge node [right] {$ $} (xa4)
(xa2) edge node [right] {$ $} (xa3)
(xa2) edge node [right] {$ $} (xa4)
(xa2) edge node [right] {$ $} (xb6)
(xa3) edge node [right] {$ $} (xa4)
(xa1) edge node [right] {$ $} (xb1)
(xa2) edge node [right] {$ $} (xb2)
(xa3) edge node [right] {$ $} (xb3)
(xa4) edge node [right] {$ $} (xb4)
(xa1) edge node [right] {$ $} (xb5)
(xb1) edge node [right] {$ $} (xb2)
(xb1) edge node [right] {$ $} (xb3)
(xb1) edge node [right] {$ $} (xb4)
(xb1) edge node [right] {$ $} (xb5)
(xb1) edge node [right] {$ $} (xb6)
(xb2) edge node [right] {$ $} (xb3)
(xb2) edge node [right] {$ $} (xb4)
(xb2) edge node [right] {$ $} (xb5)
(xb2) edge node [right] {$ $} (xb6)
(xb3) edge node [right] {$ $} (xb4)
(xb3) edge node [right] {$ $} (xb5)
(xb3) edge node [right] {$ $} (xb6)
(xb4) edge node [right] {$ $} (xb5)
(xb4) edge node [right] {$ $} (xb6)
(xb5) edge node [right] {$ $} (xb6);
\node (ya1) at (3.25,0.75) {$a_1$};
\node (ya2) at (3.25,0.25) {$a_2$};
\node (ya3) at (2.75,1) {$a_3$};
\node (ya4) at (2.75,0) {$a_4$};
\node (ya5) at (2.25,.5) {$a_5$};
\node (yb1) at (3.75,0.7) {$b_1$};
\node (yb2) at (3.75,0.3) {$b_2$};
\node (yb3) at (4.25,1) {$b_3$};
\node (yb4) at (4.25,0) {$b_4$};
\node (yb5) at (5.25,.5) {$b_5$};
\node (yb6) at (4.75,.9) {$b_6$};
\node (yb7) at (4.75,.1) {$b_7$};
\path[font=\small,>=angle 90]
(ya1) edge node [right] {$ $} (ya2)
(ya1) edge node [right] {$ $} (ya3)
(ya1) edge node [right] {$ $} (ya4)
(ya1) edge node [right] {$ $} (ya5)
(ya2) edge node [right] {$ $} (ya3)
(ya2) edge node [right] {$ $} (ya4)
(ya2) edge node [right] {$ $} (ya5)
(ya3) edge node [right] {$ $} (ya4)
(ya3) edge node [right] {$ $} (ya5)
(ya4) edge node [right] {$ $} (ya5)
(ya1) edge node [right] {$ $} (yb1)
(ya1) edge node [right] {$ $} (yb6)
(ya2) edge node [right] {$ $} (yb2)
(ya2) edge node [right] {$ $} (yb7)
(ya3) edge node [right] {$ $} (yb3)
(ya4) edge node [right] {$ $} (yb4)
(ya5) edge node [right] {$ $} (yb5)
(yb1) edge node [right] {$ $} (yb2)
(yb1) edge node [right] {$ $} (yb3)
(yb1) edge node [right] {$ $} (yb4)
(yb1) edge node [right] {$ $} (yb5)
(yb1) edge node [right] {$ $} (yb6)
(yb1) edge node [right] {$ $} (yb7)
(yb2) edge node [right] {$ $} (yb3)
(yb2) edge node [right] {$ $} (yb4)
(yb2) edge node [right] {$ $} (yb5)
(yb2) edge node [right] {$ $} (yb6)
(yb2) edge node [right] {$ $} (yb7)
(yb3) edge node [right] {$ $} (yb4)
(yb3) edge node [right] {$ $} (yb5)
(yb3) edge node [right] {$ $} (yb6)
(yb3) edge node [right] {$ $} (yb7)
(yb4) edge node [right] {$ $} (yb5)
(yb4) edge node [right] {$ $} (yb6)
(yb4) edge node [right] {$ $} (yb7)
(yb5) edge node [right] {$ $} (yb6)
(yb5) edge node [right] {$ $} (yb7)
(yb6) edge node [right] {$ $} (yb7);
\end{tikzpicture}
$
    \caption{The graphs $\SZ{4}{2}$ and $\SZ{5}{2}$}
    \label{figkn0}
\end{figure}

\begin{proposition}\label{propkn0}
The graph $\SZ{k}{n}$ does not occur as the prime character degree graph of any solvable group for all $k\geq4$ and $1\leq n\leq k$.
\end{proposition}
\begin{proof}
We proceed by induction on $n$.
Our base case ($n=1$), has been proven in Lemma \ref{lemk10}. In particular, we have shown that $\SZ{k}{1}$ does not occur as the prime character degree graph of any solvable group for all $k\geq4$.

Next, for our inductive hypothesis: given any integer $w\geq1$, we assume that the graph $\SZ{k}{w}$ does not occur as the prime character degree graph of any solvable group for all $k\geq4$. Of course, this actually holds for all $k\geq\max\{4,w\}$.

We will complete our inductive step by showing that the subsequent graph $\SZ{k}{w+1}$ also does not occur. Following the strategy that we just witnessed in Lemma \ref{lemk10}, we will prove that every vertex is admissible.

First, we let $k\in\mathbb{N}$ such that $k\geq\max\{4,w+1\}$. Second, we notice that the vertex set associated with the graph $\SZ{k}{w+1}$ can be expressed as the disjoint union of $\alpha$, $\overline{\alpha}$, $\beta$, and $\overline{\beta}$ where
\begin{align*}
\alpha&=\{a_1,\ldots,a_{w+1}\},\\
\overline{\alpha}&=\{a_{w+2},\ldots,a_k\},\\
\beta&=\{ b_1,\ldots,b_{w+1},b_{k+1},\ldots,b_{k+w+1}\},\text{~and}\\
\overline{\beta}&=\{b_{w+2},\ldots,b_k\}.
\end{align*}
Furthermore, we notice that every vertex in each set is symmetric with every other vertex in that set. Finally, if indeed $k=w+1$, then both $\overline{\alpha}$ and $\overline{\beta}$ are empty.

We first consider the vertex $a_1$. Losing the vertex $a_1$ and all incident edges yields a graph that is isomorphic to $\SRi{k-1}{w}{2}$, which was proven not to occur in Proposition \ref{propknmR}. We next consider losing edges incident to the prime $a_1$. The edge $\epsilon(a_1,a_i)$ ($2\leq i\leq k$) cannot be lost without violating P\'alfy's condition with the vertices $a_1$, $a_i$, and $b_j$ for some $2\leq j\leq k$ with $i\neq j$. Next, losing either of the edges $\epsilon(a_1,b_1)$ or $\epsilon(a_1,b_{k+1})$ gives us the graph $\SR{k}{w}$, which does not occur by Theorem \ref{Rfam}. Finally, losing both of $\epsilon(a_1,b_1)$ and $\epsilon(a_1,b_{k+1})$ results in a graph of diameter three which is in contradiction to Theorem \ref{diameter}\eqref{diam3}. Hence, $a_1$ is admissible, and by symmetry, so is every vertex in $\alpha$.

Next, we investigate the vertex $a_{w+2}$. The graph resulting from the loss of $a_{w+2}$ and all incident edges is $\SR{k-1}{w+1}$, known not to occur by way of Theorem \ref{Rfam}. Looking at losing edges incident to $a_{w+2}$, we see that the removal of the edge $\epsilon(a_{w+2},a_i)$ for any $1\leq i\leq k$ with $i\neq w+2$ will violate P\'alfy's condition as before, and the loss of the edge $\epsilon(a_{w+2},b_{w+2})$ results in a diameter three graph, which will once again be handled by the main result from \cite{S}. Hence, $a_{w+2}$ is admissible, and by symmetry, so is every vertex in $\overline{\alpha}$.

We next consider the vertex $b_1$, whose removal along with incident edges yields the graph $\SZ{k}{w}$, which does not occur by our inductive hypothesis as stated above. One can also see that edges incident to $b_1$ cannot be lost. In fact, the removal of the edge $\epsilon(b_1,b_i)$ for any $2\leq i\leq k+w+1$ results in a graph that does not occur due to P\'alfy's condition with $b_1$, $b_i$, and $a_j$ for some $2\leq j\leq k$ with $j \not\in\{i,i-k\}$. Furthermore, the edge $\epsilon(b_1,a_1)$ was proven essential previously. Thus, $b_1$ is admissible. Again due to symmetry, we get that every vertex in $\beta$ is also admissible.

Finally, we consider $b_{w+2}$. Induced by the loss of $b_{w+2}$ and all incident edges, the resulting graph $\SL{k-1}{w+1}$ does not occur by Theorem \ref{Lfam}. For the second part of Definition \ref{adstrong}, we see that the loss of edges incident to $b_{w+2}$ will either violate P\'alfy's as above or will result in a violation of a diameter three graph. Regardless, we see that we cannot lose those edges, and therefore we conclude that $b_{w+2}$ is admissible. We therefore get that every vertex in $\overline{\beta}$ is admissible by a symmetric argument.

Since every vertex is admissible, then the graph $\SZ{k}{w+1}$ does not occur by Theorem \ref{admissible}. Hence, by induction, the graph $\SZ{k}{n}$ does not occur as the prime character degree graph of any solvable group for all $k\geq4$ and $1\leq n\leq k$.
\end{proof}

Uniting all our results from this section, we are finally able to state our main theorem for the family $\{\SZ{k}{n}\}$.

\begin{thmC}
The graph $\SZ{k}{n}$ occurs as the prime character degree graph of a solvable group when $(k,n)\in\{(1,1), (2,1), (2,2)\}$. Otherwise, $\SZ{k}{n}$ does not occur as the prime character degree graph of any solvable group.
\end{thmC}
\begin{proof}
Combine the result from Proposition \ref{prop1and20}, which handles the cases of $k=1$ and $k=2$, along with the results from Propositions \ref{prop310} and \ref{prop320} and Corollary \ref{cor330} (the case of $k=3$), and the result from Proposition \ref{propkn0} (the case of $k\geq4$).
\end{proof}

\end{document}